\documentclass[a4paper,11pt,reqno]{amsart}

\usepackage{fullpage}

\usepackage{soul}

\usepackage{dsfont,amssymb,amsmath,amsfonts,amsthm,mathtools,scalerel,color}

\usepackage{hyperref}

\usepackage[alphabetic,numeric,abbrev,msc-links]{amsrefs}

\DefineSimpleKey{bib}{how}

\renewcommand{\eprint}[1]{#1}
\BibSpec{misc}{%
  +{}{\PrintAuthors}  {author}
  +{,}{ \textit}      {title}
  +{,}{ }             {how}
  +{}{ \parenthesize} {date}
  +{,} { available at \eprint}        {eprint}
  +{,}{ available at \url}{url}
  +{,}{ }             {note}
  +{.}{}              {transition}
}

\numberwithin{equation}{section}

\newtheorem{theorem}{Theorem}[section]
\newtheorem{corollary}[theorem]{Corollary}

\newtheorem{lemma}[theorem]{Lemma}
\newtheorem{proposition}[theorem]{Proposition}
\theoremstyle{remark}

\newtheorem{example}[theorem]{Example}
\theoremstyle{definition}

\newtheorem{definition}[theorem]{Definition}

\newcommand\bp{\begin{proof}}
\newcommand\ep{\end{proof}}
\newcommand\ee{\nopagebreak\mbox{\ }\hfill$\diamond$}

\newcommand\bicrossl{%
  \mathrel{\scalerel*{\mathrel{\triangleright}\joinrel\blacktriangleleft}{x}}}

\newcommand\Dhat{{\hat\Delta}}

\newcommand\Ad{\operatorname{Ad}}

\newcommand\Aut{\operatorname{Aut}}

\newcommand\Mat{\operatorname{Mat}}
\newcommand\Res{\operatorname{Res}}
\newcommand\Tr{\operatorname{Tr}}

\newcommand\Op{\operatorname{Op}}
\newcommand\HS{{\rm HS}}

\newcommand{\C}{{\mathbb C}}
\newcommand{\R}{{\mathbb R}}
\newcommand\T{{\mathbb T}}
\newcommand\Z{{\mathbb Z}}

\newcommand{\N}{{\mathcal N}}
\newcommand\OO{\mathcal{O}}
\newcommand{\G}{\mathcal{G}}

\newcommand\kk{\mathbf{k}}

\newcommand{\vf}{\varphi}

\newcommand\pent{\mathrm{pent}}

\newcommand{\CL}{{\mathcal L}{}}

\newcommand{\CJ}{\mathcal J}

\newcommand{\CU}{\mathcal U}
\newcommand{\CF}{\mathcal F}

\begin{document}

\date{December 1, 2023}

\author[Bieliavsky]{Pierre Bieliavsky}
\email{Pierre.Bieliavsky@uclouvain.be}
\address{Institut de Recherche en Math\'ematique et Physique, Universit\'e Catholique de Louvain, Chemin du Cyclotron, 2, 1348 Louvain-la-Neuve, Belgium}

\author[Gayral]{Victor Gayral}
\email{victor.gayral@univ-reims.fr}
\address{Laboratoire de Math\'ematiques, CNRS UMR 9008, Universit\'e de Reims Champagne-Ardenne,
Moulin de la Housse - BP 1039,
51687 Reims, France}

\author[Neshveyev]{Sergey Neshveyev}
\email{sergeyn@math.uio.no}
\address{Department of Mathematics, University of Oslo,
P.O. Box 1053 Blindern, NO-0316 Oslo, Norway}
\thanks{Supported by the NFR project 300837 ``Quantum Symmetry''}

\author[Tuset]{Lars Tuset}
\email{larst@oslomet.no}
\address{Department of Computer Science, OsloMet - storbyuniversitetet,
P.O. Box 4 St. Olavs plass, NO-0130 Oslo, Norway}

\title{Quantization of locally compact groups associated with essentially bijective \texorpdfstring{$1$}{1}-cocycles}

\begin{abstract}
Given an extension $0\to V\to G\to Q\to1$ of locally compact groups, with $V$ abelian, and a compatible essentially bijective $1$-cocycle $\eta\colon Q\to\hat V$, we define a dual unitary $2$-cocycle on $G$ and show that the associated deformation of $\hat G$ is a cocycle bicrossed product defined by a matched pair of subgroups of $Q\ltimes\hat V$. We also discuss an interpretation of our construction from the point of view of Kac cohomology for matched pairs.

Our setup generalizes that of Etingof and Gelaki for finite groups and its extension due to Ben David and Ginosar, as well as our earlier work on locally compact groups satisfying the dual orbit condition. In particular, we get a locally compact quantum group from every involutive nondegenerate set-theoretical solution of the Yang--Baxter equation, or more generally, from every brace structure.

On the technical side, the key new points are constructions of an irreducible projective representation of $G$ on $L^2(Q)$ and a unitary quantization map $L^2(G)\to\HS(L^2(Q))$ of Kohn--Nirenberg type.
\end{abstract}

\maketitle

\section*{Introduction}

The present work is a continuation of our earlier paper~\cite{BGNT3} and its follow-up~\cite{GM} by the second author and Marie. In~\cite{BGNT3} we introduced the following notion. Assume a locally compact group~$Q$ acts on an abelian locally compact group $V$. We say that the pair $(V,Q)$ satisfies the \emph{dual orbit condition} if there is a point $\xi_0\in\hat V$ such that the map $Q\to \hat V$, $q\mapsto q^\flat\xi_0$, is a measure class isomorphism. When $Q$ is a Lie group and $V\cong\R^n$, then the map $Q\to\hat V$ is injective and the orbit $Q^\flat\xi_0\subset\hat V$ is automatically open, which implies that $G:=Q\ltimes V$ is a Frobenius Lie group~\cite{MR0556091}, hence a Poisson--Lie group. The problem at hand was to quantize $G$ in the analytic setting by constructing a unitary dual $2$-cocycle on $G$, that is, a unitary $\Omega\in W^*(G)\bar\otimes W^*(G)$ satisfying the identity
$$
(\Omega\otimes1)(\Dhat\otimes\iota)(\Omega)=(1\otimes\Omega)(\iota\otimes\Dhat)(\Omega).
$$

More generally, one can study such elements $\Omega$ beyond the Lie group case and independently of the quantization problem. This leads to examples of quantum groups: every $\Omega$ defines a locally compact quantum group $G_\Omega$ with von Neumann bialgebra $(W^*(G),\Omega\Dhat(\cdot)\Omega^*)$. In~\cite{BGNT3} we showed, under the dual orbit condition, how to explicitly construct at least one such $\Omega$. In~\cite{GM} a family of $\Omega$'s parameterized by the $2$-cocycles on $Q$ was constructed.

It is clear that the dual orbit condition is never satisfied for nontrivial finite groups. Nevertheless, as we subsequently realized, our setting is close to a finite group setup considered earlier by Etingof and Gelaki for the same purpose of constructing examples of dual cocycles~\cite{EG}. Namely, assume that in addition to an action of $Q$ on $\hat V$ we have a $1$-cocycle $\eta\colon Q\to\hat V$. It is called a \emph{bijective $1$-cocycle} if it is bijective as a map~\cites{MR1722951,EG}. In our analytical setting it is natural to consider the weaker notion of an \emph{essentially bijective $1$-cocycle}, where we require only that $\eta\colon Q\to\hat V$ is a measure class isomorphism. The dual orbit condition for $\xi_0\in\hat V$ is then equivalent to essential bijectivity of the $1$-coboundary $\eta(q):=q^\flat\xi_0-\xi_0$.

Replacing the dual orbit condition by essential bijectivity of a $1$-cocycle significantly expands the class of groups that we can analyze, see Examples~\ref{ex:cd}-\ref{ex:dc} and references therein. Let us also mention that for Lie groups one has the following known geometric interpretation of essential bijectivity: if $Q$ is a Lie group and $V\cong\R^n$, then every essentially bijective $1$-cocycle $\eta\colon Q\to\hat V$ is injective and has open image, hence it defines a $Q$-invariant affine structure on~$Q$~\cite{MR1272113}.

The work of Etingof and Gelaki was later generalized by Ben David and Ginosar~\cite{MR2534251}, who considered extensions $0\to V\to G\to Q\to1$ instead of
semidirect products $G=Q\ltimes V$. In this case, to construct a dual $2$-cocycle on $G$ one needs a compatibility condition between the bijective cocycle $\eta$ and the extension class. The same condition makes sense in our analytical setting of locally compact groups and essentially bijective cocycles, see Definition~\ref{def:compatible}. With this condition we are able to construct a dual unitary $2$-cocycle on $G$ and identify the corresponding locally compact quantum groups, generalizing the main results of~\cites{BGNT3,GM}.

Let us say a few words about how our approach to dual $2$-cocycles compares to~\cites{EG,MR2534251}. In the finite group case there are several ways of describing nondegenerate dual $2$-cocycles on~$G$, where by nondegeneracy one means that a cocycle is not induced from a proper subgroup. In addition to explicit formulas, as in~\cite{EG}, one can also describe them in terms of usual nondegenerate $2$-cocycles on $G$, which is the approach taken in~\cite{MR2534251}. But in our opinion very often the most transparent way to define them is to start with an irreducible projective representations of dimension~$\sqrt{|G|}$. For infinite groups this condition has to be properly formulated as an analytic version of the Hopf--Galois condition. This has been done by De Commer in~\cite{DC}, and this is the approach we take.

\smallskip

In more detail, the paper is organized as follows. After a short preliminary section, we define the key ingredients of our approach in Section~\ref{sec:rep}. First of all it is a square integrable irreducible projective representation $\pi\colon G\to PU(L^2(Q))$ such that $(B(L^2(Q)),\Ad\pi)$ is a $G$-Galois object. The definition of $\pi$ seems to us very natural, and all the required properties of $\pi$ are verified in a straightforward way. Second, it is a \emph{quantization map}, a $G$-equivariant unitary operator $L^2(G)\to \HS(L^2(Q))$. Its construction is motivated by the Kohn--Nirenberg quantization used in~\cite{BGNT3} and an integral formula for it suggested in~\cite{GM}. The strategy is therefore similar to \cites{BGNT3,GM}, but we streamline some of the arguments in both papers. In particular, the verification of the Galois property no longer relies on deformations of function algebras.

Section~\ref{sec:2mult} contains our main results about the dual $2$-cocycle $\Omega$ on $G$ arising from $\pi$. We obtain an explicit formula for $\Omega$, compute the multiplicative unitary of the quantum group $\hat G_\Omega:=(W^*(G),\Omega\Dhat(\cdot)\Omega^*)$ and recognize  it as the multiplicative unitary of a cocycle bicrossed product defined by the matched pair $(Q,Q_\eta)$ of subgroups of $Q\ltimes\hat V$, where $Q_\eta:=\{(q,\eta(q))\mid q\in Q\}\cong Q$.

Therefore $\hat G_\Omega$ fits into a short exact sequence
\begin{equation}\label{eq:exte}
1\to\hat Q\to \hat G_\Omega\to Q_\eta\to 1
\end{equation}
of locally compact quantum groups. Such extensions are classified by the second Kac cohomology of matched pairs~\cites{Kac,MR2115374}. The construction of $\pi$ and $\Omega$ depends on some choices, so we in fact get a family of dual cocycles $\Omega$, giving rise to a family of extensions~\eqref{eq:exte} parameterized by the classes in $H^2(Q;\T)$. We use Kac's exact sequence to show that when $\eta$ is a coboundary, this parameterization is one-to-one (but not all extensions~\eqref{eq:exte} are obtained this way in general). This answers a question left open in~\cite{GM}.

Finally, the paper has two appendices that contain details of a number of computations in Section~\ref{sec:2mult}, which are too tiresome to present  in the main text.

\bigskip

\section{Setup}

Throughout the paper $Q$ and $V$ denote second countable locally compact groups
with $V$ abelian. We assume that we are given a continuous homomorphism $Q\to\Aut(V)$ and consider an extension~$G$  of~$V$ by~$Q$:
\begin{equation}\label{eq:extension}
0\to V\to G\to Q\to1.
\end{equation}
Let $\beta\in Z^2(Q;V)$ be a representative of the class of this extension, where we consider Moore cohomology with Borel cochains~\cite{MR0414775}. Thus, $\beta\colon Q\times Q\to V$ is a Borel map satisfying
$$
\beta(q_1,q_2)+\beta(q_1q_2,q_3)=q_1\beta(q_2,q_3)+\beta(q_1,q_2q_3).
$$
Then the group $G$
can be realized  as the set $Q\times V$ with group law
$$
(q_1,v_1)(q_2,v_2)=(q_1q_2,v_1+q_1v_2+\beta(q_1,q_2)).
$$
The inverse is given by
$$
(q,v)^{-1}=(q^{-1},-q^{-1}v-\beta(q^{-1},q))=(q^{-1},-q^{-1}v-q^{-1}\beta(q,q^{-1})).
$$
We remark that the Borel structure on $G$ is that of the product-space $Q\times V$, but the topology is in general different.

Denote by $\hat V$ the dual abelian group endowed with the dual action of $Q$, which we denote by~$\flat$.
By a $1$-cocycle $\eta\in Z^1_c(Q;\hat V)$ we mean a continuous map $\eta\colon Q\to\hat V$ satisfying
$$
\eta(q_1q_2)=\eta(q_1)+q_1^\flat\eta(q_2).
$$
It is a $1$-coboundary if it has the form $\partial\xi_0(q):=q^\flat\xi_0-\xi_0$ for some $\xi_0\in\hat V$. We remark that we could formally weaken the continuity condition on $\eta$ by assuming that it is only a Borel map, but this would lead to the same notion, since $\eta$ defines a homomorphism from $Q$ into the group $Q\ltimes\hat V$, $q\mapsto(q,\eta(q))$, and any Borel homomorphism of Polish groups is automatically continuous.

We will write the duality pairing between  $\hat V$ and $V$ as $e^{i \langle\cdot,\cdot\rangle}$. This is just a notation, we do not mean that we have an $\R$-valued pairing between  these groups, although this is indeed the case in many examples. Consider then the continuous map $\omega_\eta:G\times G\to \T$ given by
\begin{equation}\label{eq:omega-eta}
\omega_\eta(g_1,g_2):=e^{i \langle\eta(q_1^{-1}),v_2\rangle}=e^{-i \langle\eta(q_1),q_1v_2\rangle}.
\end{equation}
It is a $2$-cocycle on $G$ when $\beta=0$, since
$$
\partial \omega_\eta(g_1,g_2,g_3)=\frac{ \omega_\eta(g_2,g_3) \omega_\eta(g_1,g_2g_3)}{ \omega_\eta(g_1,g_2) \omega_\eta(g_1g_2,g_3)}
=e^{i\langle\eta(q_1^{-1}),\beta(q_2,q_3)\rangle}.
$$

Let  $\psi \in Z^3(Q;\T)$ be the $3$-cocycle on $Q$ given by
\begin{equation}\label{eq:psi}
\psi (q_1,q_2,q_3):=e^{i\langle\eta(q_1^{-1}),\beta(q_2,q_3)\rangle}=e^{-i\langle\eta(q_1),q_1\beta(q_2,q_3)\rangle}.
\end{equation}
This cocycle is a representative of the cup-product class $-[\eta]\cup[\beta]\in H^3(Q;\T)$ and,
unless $[\beta]=[0]\in H^2(Q;V)$ or $[\eta]=[0]\in H^1(Q;\hat V)$, there is no reason why it should be trivial on $Q$, only its inflation to $G$ is trivial.
We note in passing that if $\beta=\partial \chi$ then $\psi=\partial b$ for $b(q_1,q_2)=e^{-i\langle \eta(q_1^{-1}),\chi(q_2)\rangle}$ and when
$\eta=\partial\xi_0$ then $\psi=\partial b$ for $b(q_1,q_2)=e^{-i\langle \xi_0,\beta(q_1,q_2)\rangle}$.
\begin{definition}\label{def:compatible}
A $1$-cocycle $\eta\in Z^1_c(Q;\hat V)$ is called {\em essentially bijective}, if the map $\eta\colon Q\to \hat V$ is a measure class isomorphism.
A pair of cocycles $(\eta,\beta)\in Z^1_c(Q;\hat V)\times Z^2(Q;V)$ is called {\em compatible} if we have  $[\psi]=[1]\in H^3(Q;\T)$.
\end{definition}

Obviously, the compatibility condition is satisfied if $[\eta]=[0]$ or $[\beta]=[0]$. In particular, if $\eta$ is trivial, then we can consider any extension.

Note that essential bijectivity is a property of a given cocycle, not of its cohomology class. In order to see this, consider a coboundary $\partial\xi_0$. It is essentially bijective if and only the map $Q\to \hat V$, $q\mapsto q^\flat\xi_0$, is a measure class isomorphism, a property that we called a dual orbit condition in~\cite{BGNT3}. It is obvious that  for nontrivial groups this property is never satisfied for $\xi_0=0$.

Nevertheless essential bijectivity is almost a cohomological invariant in the following measure theoretic sense.

\begin{lemma}
If $\eta\in Z^1_c(Q;\hat V)$ is an essentially bijective $1$-cocycle, then the cohomologous cocycle $\eta+\partial\xi_0$ is essentially bijective for almost all $\xi_0\in\hat V$.
\end{lemma}

\bp
If $\eta$ is a measure class isomorphism, so is the map $\eta(\cdot\, q_0)$, and hence also the map $\eta(\cdot\,q_0)-\eta(q_0)$, for every $q_0\in Q$. By the cocycle identity this means that the cocycle $\eta+\partial\xi_0$ is essentially bijective for $\xi_0=\eta(q_0)$ and all $q_0\in Q$. As such $\xi_0$ form a set of full measure by assumption, this proves the lemma.
\ep

Given $\eta\in Z^1_c(Q;\hat V)$, consider the closed subgroup $Q_\eta:=\big\{(q,\eta(q))\mid q\in Q\}$ of the semidirect product $Q\ltimes\hat V$.
Of course, $Q_\eta$ is always isomorphic to $Q$, but it is conjugate to $Q$ in $Q\ltimes\hat V$ if and only if $\eta$ is a coboundary.
We have the following characterization of essentially bijective $1$-cocycles in terms of matched pairs.

\begin{lemma}\label{lem:matched}
A cocycle $\eta\in Z^1_c(Q;\hat V)$ is essentially bijective  if and only if $(Q_\eta,Q)$ is a matched pair of subgroups of $Q\ltimes\hat V$, that is, the product map
$Q_\eta\times Q\to Q\ltimes\hat V$ is a measure class isomorphism.
\end{lemma}

\begin{proof}
Observe that in  $Q\ltimes\hat V$ we have
$$
(q_1,\eta(q_1))(q_2,0)=(q_1q_2,\eta(q_1))\quad\text{for all}\quad q_1,q_2\in Q.
$$
This shows that the product map $Q_\eta\times Q\to Q\ltimes\hat V$ defines a measure class isomorphism if and only if $\eta\colon Q\to\hat V$ does.
\end{proof}

Besides  the abelian case $G=\hat V\times V$ (where one takes $\eta={\rm Id}_{\hat V}$) and the finite group case of~\cites{EG,MR2534251}, the following
examples fit into our setting.

\begin{example}[Dual orbit condition]
Assume that a pair $(Q,V)$ satisfies the already mentioned dual orbit condition of \cite{BGNT3}, that is, there exists $\xi_0\in\hat V$ such that the map
$Q\to\hat V$, $q\mapsto q^\flat \xi_0$, is a measure class isomorphism. In this case, the coboundary $\eta(q):=\partial\xi_0(q)=q^\flat\xi_0-\xi_0$, is an essentially bijective $1$-cocycle
and every $\beta\in Z^2(Q;V)$ is compatible with $\eta\in Z^1_c(Q;\hat V)$. We refer the reader to~\cite{BGNT3}*{Section~2.4} for many concrete examples.
\end{example}

\begin{example}[$Q$ compact and $V$ discrete]\label{ex:cd}
Assume $\tilde V$ is a locally compact abelian group and let $K\subset\tilde V$ and $Q\subset\Aut(\tilde V)$ be closed subgroups such that
$Q$ leaves $K$ globally invariant. Then, setting $V=\tilde V/K$, we get a homomorphism $Q\to\Aut( V)$.
Assume there exists  $\xi_0\in \hat{\tilde V}$ such that
the map $\tilde\eta \colon Q\to \hat{\tilde V}$, $ q\mapsto q^\flat \xi_0-\xi_0$, takes values in the annihilator $K^\perp$ of $K$ in $\hat{\tilde V}$ and that
$\tilde\eta\colon Q\to K^\perp$ defines a measure class isomorphism.
Under the identification  $\hat V\simeq K^\perp$, we get an essentially bijective $1$-cocycle $\eta\in Z^1 (Q,\hat V)$. Note that
$$
\omega_\eta(g_1,g_2)=e^{i \langle {q_1^{-1}}^\flat\xi_0-\xi_0,v_2\rangle},
$$
and this $2$-cocycle for $Q\ltimes V$ can be nontrivial when $\xi_0\in \hat{\tilde V}\setminus K^\perp$.

This covers the  examples studied by Jondreville \cite{J}, where  $Q$ is compact and $K$ is open-compact, and so $V$ is discrete, but we  no longer need
his assumptions of commutativity  of the local field and of odd cardinality of the residual field.
Namely, let $\kk$ be a non-Archimedean local skew-field, $\OO$ its maximal compact subring, $\varpi$ a uniformizer of~$\OO$, and let $\Psi$ be a character of~$\kk$ trivial on~$\OO$
but not on $\varpi^{-1}\OO$. We can identify~$\hat\kk$ with~$\kk$ using the pairing $(x,y)=\Psi(xy)$. Under this identification we have
$(\varpi^{-\ell}\OO)^\perp=\varpi^\ell\OO$ for all $\ell\in\Z$.
Put $\tilde V=\kk$, $Q=1+\varpi^\ell\OO$ and $K=\varpi^{-\ell}\OO$ for some $\ell\ge1$. Since
$$
\hat V=\widehat{\kk/\varpi^{-\ell}\OO}\simeq (\varpi^{-\ell}\OO)^\perp= \varpi^\ell\OO,
$$
we see that our requirements are satisfied. Indeed, it suffices to take $\xi_0=1\in\kk$,
which does not belong to  $\hat V=\varpi^\ell\OO$, to get that the map $\eta\colon 1+\varpi^\ell\OO \to \varpi^\ell\OO $, $ q\mapsto q^{-1}-1$, is indeed a measure class isomorphism.

More generally, we can take
$$
V=\Mat_n(\kk)/\varpi^{-\ell}\Mat_n(\OO),\quad Q={\rm I}_n+\varpi^{\ell}\Mat_n(\OO),
$$
or even more generally,
$$
V=\underbrace{\operatorname{Mat}_n(\kk)/\varpi^{-\ell}\Mat_n(\OO)\oplus \cdots\oplus\operatorname{Mat}_n(\kk)/\varpi^{-\ell}\Mat_n(\OO)}_m ,
$$
and
$$ \quad Q=
\begin{pmatrix}{\rm I}_n&\cdots&0&0\\
\vdots& \ddots& \vdots&\vdots\\
0&\cdots&{\rm I}_n&0\\
\varpi^{\ell}\Mat_n(\OO)&\cdots&\varpi^{\ell}\Mat_n(\OO)&{\rm I}_n+\varpi^{\ell}\Mat_n(\OO)\end{pmatrix}\subset{\rm GL}_{nm}({\bf k}).
$$
\end{example}

\begin{example}[$Q$ discrete and $V$ compact]\label{ex:dc}
Assume $Q$ is discrete and $V$ is compact, so $\hat V$ is discrete. In this case a cocycle $\eta$ is essentially bijective if and only if it is bijective. By~\cite{MR2298848}, to have an action of some discrete group $Q$ on $V$ and a bijective $1$-cocycle $\eta\colon Q\to\hat V$ is the same as to have a brace structure on $\hat V$. The latter means that on $\hat V$ we are given one more binary operation $\hat V\times\hat V\to\hat V$, $(a,b)\mapsto ab$, such that we have right distributivity
$$
(a+b)c=ac+bc
$$
and such that $\hat V$ is a group with respect to the operation $\circ$ defined by
$$
a\circ b:=ab+a+b.
$$
Namely, if $\hat V$ is a brace, we let $Q=(\hat V,\circ)^{\mathrm{op}}$, $\eta\colon Q\to\hat V$ be the identity map and define an action of $Q$ on $\hat V$ by $a^\flat b:=ba+b$.

For examples of bijective cocycles and braces, see \citelist{\cite{MR1722951}\cite{MR2298848}\cite{MR3763276}}.\ee
\end{example}

Next, let us discuss a few measure theoretic aspects of our setup. The Haar measure and the modular function on $G$ are the same as for $Q\ltimes V$. Namely, let $|q|=|q|_V$ be the modulus of the action of $q\in Q$ on $V$, so that if $dv$ is a Haar measure on $V$, then
$$
\int_V f(qv)\,dv=|q|^{-1}\int_V f(v)\,dv\quad\text{for all}\quad f\in C_c(V).
$$
Let $dq$ be a left Haar measure on $Q$ and let $\Delta_Q$  be the
modular function. Then a left invariant Haar  measure and the modular function on $G$ are respectively given by
$$
d(q,v)=\frac{dq\,dv}{|q|}\quad\mbox{and}\quad \Delta(q,v)=\frac{\Delta_Q(q)}{|q|}.
$$

\begin{lemma}\label{lem:normalization}
Fix a Haar measure on $\hat V$. Given an essentially bijective $1$-cocycle $\eta\in Z^1_c(Q;\hat V)$, the Haar measure on $Q$ can be normalized so that
$$
\int_Q f(\eta(q))\,\frac{dq}{|q|}=\int_{\hat V}f(\xi)\,d\xi
$$
for all positive Borel functions $f$ on $\hat V$.
\end{lemma}

\bp
We have an action of $Q\ltimes V$ on $V\times\hat V$ defined by
$$
(q,v)(w,\xi):=(qw+v,q^\flat\xi+\eta(q)).
$$
The Haar measure on $V\times\hat V$ is invariant under this action. It follows that the pull-back of this measure under the measure class isomorphism $Q\ltimes V\to V\times\hat V$, $(q,v)\mapsto (q,v).(0,0)=(v,\eta(q))$, is a nonzero left-invariant Radon measure on the group $Q\ltimes V$. Hence it coincides with $|q|^{-1}dq\,dv$ up to a scalar factor, which proves the lemma.
\ep

We assume from now on that this is how the Haar measure on $Q$ is normalized. We also normalize the Haar measure on $V$ so that the Fourier transform  $\CF_V$ defined by
$$
(\CF_Vf)(\xi):=\int_V e^{-i\langle\xi,v\rangle} f(v)\,dv
$$
is a unitary from $L^2(V)$ onto $L^2(\hat V)$. Altogether this uniquely determines the Haar measure on~$G$: if we multiply the Haar measure on~$\hat V$ by a factor, then the Haar measure on $Q$ also has to be multiplied by this factor, while the Haar measure on $V$ has to be divided by the same factor, and therefore the Haar measure on~$G$ remains unchanged.

\bigskip

\section{Projective representations, Galois objects and quantization maps}\label{sec:rep}

Throughout this section we assume that $(\eta,\beta)\in Z^1_c(Q;\hat V)\times Z^2(Q;V)$ is a compatible pair of cocycles, with $\eta$ essentially bijective, and $G$ is the extension of $V$ by $Q$ defined by $\beta$. Fix a Borel map $b \colon Q\times Q\to \T$ such
that $\psi =\partial b$, where~$\psi$ is given by~\eqref{eq:psi}. Define a $2$-cocycle $\omega\in Z^2(G;\T)$ by
\begin{equation}\label{eq:omega}
\omega(g_1,g_2):=\overline{\omega_\eta(g_1,g_2)}\,b (q_1,q_2),
\end{equation}
where $\omega_\eta$ is given by~\eqref{eq:omega-eta}.  As was already mentioned, if
$\eta=\partial\xi_0$, then we can take $b(q_1,q_2)=e^{-i\langle \xi_0,\beta(q_1,q_2)\rangle}$. In this case,  we get $\omega=\partial h$ for $h(g)=e^{i\langle \xi_0,v \rangle}$.

 The following construction is a key to everything that follows.

\begin{proposition}\label{prop:representation}
Define a map $\pi\colon G\to U(L^2(Q))$ by
\begin{align}
\label{def-pi}
(\pi(q,v)\vf)(q_0):=e^{-i \langle\eta(q_0),v\rangle}\,b (q_0^{-1},q)\,\vf\big(q^{-1}q_0\big),\quad \vf\in L^2(Q).
\end{align}
This is a projective unitary representation with cocycle $\omega$, so that $\pi(g_1)\pi(g_2)=\omega(g_1,g_2)\,\pi(g_1g_2)$ for all $g_1,g_2\in G$.
Furthermore, $\pi$ is irreducible, square-integrable and the corresponding Duflo--Moore operator is the operator of multiplication by the function $\Delta^{-1}|_Q=|\cdot|\Delta_Q^{-1}$.
\end{proposition}

Note that the restriction of $\pi$ to $V$ is just the regular representation under the identification of $L^2(V)$ with $L^2(Q)$ by means of the Fourier transform $L^2(V)\to L^2(\hat V)$ and the unitary isomorphism $L^2(\hat V)\to L^2(Q)$, $f\mapsto |\cdot|^{-1/2}f\circ\eta$. On the other hand, if $G=Q\ltimes V$ and $b\equiv1$, then the restriction of $\pi$ to $Q$ is the regular representation. Therefore for general $G$ the introduction of the factor $b (q_0^{-1},q)$ is what it takes to get a projective representation of $G$. This is the same factor that we would have used if $b$ was a $2$-cocycle and we wanted to define a projective representation of $Q$.

\begin{proof}[Proof of Proposition~\ref{prop:representation}]
We have for $\vf\in L^2(Q)$:
\begin{align*}
(\pi(g_1g_2)\vf)(q_0)&=e^{-i \langle\eta(q_0),v_1+q_1v_2+\beta(q_1,q_2)\rangle}\,b (q_0^{-1},q_1q_2)
\, \vf\big((q_1q_2)^{-1}q_0\big)\\
&=e^{-i \langle\eta(q_0),v_1+q_1v_2\rangle}\,\overline{\psi(q_0^{-1},q_1,q_2)}\,b (q_0^{-1},q_1q_2)
\, \vf\big((q_1q_2)^{-1}q_0\big)\\
&=e^{-i \langle\eta(q_0),v_1+q_1v_2\rangle}\,\overline{b (q_1,q_2)}\,
b (q_0^{-1}q_1,q_2)\,b (q_0^{-1},q_1)\, \vf\big((q_1q_2)^{-1}q_0\big),
\end{align*}
and
\begin{align*}
(\pi(g_1)\pi(g_2))\vf(q_0)&=e^{-i \langle\eta(q_0),v_1\rangle}e^{-i \langle\eta(q_1^{-1}q_0),v_2\rangle}\,b (q_0^{-1},q_1)
\,b (q_0^{-1}q_1,q_2)\,\vf\big((q_1q_2)^{-1}q_0\big)\\
&=e^{-i \langle\eta(q_0),v_1+q_1v_2\rangle}\,\overline{\omega_\eta(g_1,g_2)}\,b (q_0^{-1},q_1)
\,b (q_0^{-1}q_1,q_2)
\, \vf\big((q_1q_2)^{-1}q_0\big).
\end{align*}
This gives the first claim of the proposition.

\smallskip

Irreducibility of $\pi$ follows by a standard argument, see, e.g., \cite{GM}*{Lemma 3.1}: Since the restriction of $\pi$ to $V$ is equivalent to the regular representation, every bounded operator on~$L^2(Q)$ commuting with $\pi(V)$ is the operator $m_f$ of multiplication by a function $f\in L^\infty(Q)$. But then for $m_f$ to commute with $\pi(G)$ we need~$f$ to be invariant under the left translations by~$Q$, hence~$f$ must be constant.

\smallskip

Regarding square-integrability and the Duflo--Moore operator, we shall prove a slightly more general statement. Namely, given a Borel function  $\Psi:Q\times Q\to \T$,
consider the Borel map $G\ni g\mapsto \tilde\pi(g)\in U(L^2(Q))$ defined by
$$
\tilde\pi(q,v)\vf(q_0):=e^{-i \langle\eta(q_0),v\rangle}\,\Psi(q,q_0)\,\vf\big(q^{-1}q_0\big),\quad  \vf\in L^2(Q).
$$
We are going to prove that for all $\vf_1,\vf_2\in C_c(Q)$ we have
$$
\int_G \big|\langle\vf_1,\tilde\pi(g)\vf_2\rangle\big|^2\,dg=\|\vf_1\|_2^2\,\|\Delta^{-1/2}|_Q\,\vf_2\|_2^2.
$$

Indeed, by our normalization of Haar measures given by Lemma~\ref{lem:normalization} we have
\begin{align*}
\langle\vf_1,\tilde\pi(q,v)\vf_2\rangle&=\int \overline{\vf_1(q_0)} \,e^{-i \langle\eta(q_0),v\rangle}\,\Psi(q,q_0)\,\vf_2\big(q^{-1}q_0\big)\,
\,dq_0\\
&=\int \overline{\vf_1(\eta^{-1}(\xi))} \,e^{-i \langle\xi,v\rangle}\,\Psi(q,\eta^{-1}(\xi))\,\vf_2\big(\eta^{-1}({q^{-1}}^\flat\xi+\eta(q^{-1}))\big)\,
|\eta^{-1}(\xi)|\,d\xi\\
&=(\CF_V^* f_q)(-v),
\end{align*}
where $f_q\in L^2(\hat V)$ is defined by
$$
f_q(\xi):=\overline{\vf_1(\eta^{-1}(\xi))} \,\Psi(q,\eta^{-1}(\xi))\,\vf_2\big(\eta^{-1}({q^{-1}}^\flat\xi+\eta(q^{-1}))\big)\,
|\eta^{-1}(\xi)|.
$$
Therefore we deduce by the Plancherel theorem:
\begin{align*}
\int \big|\langle\vf_1,\tilde\pi(q,v)\vf_2\rangle\big|^2 \frac{dq\,dv}{|q|}&=\int |f_q(\xi)|^2  \frac{dq\,d\xi}{|q|}\\
&=\int |\vf_1(\eta^{-1}(\xi))|^2 \,\big|\vf_2\big(\eta^{-1}({q^{-1}}^\flat\xi+\eta(q^{-1}))\big)\big|^2|\eta^{-1}(\xi)|^2\, \frac{dq\,d\xi}{|q|}\\
&=\int |\vf_1(q_0)|^2|\vf_2(q^{-1}q_0)|^2|q^{-1}q_0|\, dq\,dq_0\\
&=\int |\vf_1(q_0)|^2|\vf_2(q^{-1})|^2\frac{dq\,dq_0}{|q|}\\
&=\int |\vf_1(q_0)|^2|\vf_2(q)|^2\frac{|q|}{\Delta_Q(q)}dq\,dq_0,
\end{align*}
which concludes the proof, since $\Delta(q)=|q|^{-1}\Delta_Q(q)$.
\end{proof}

Let $D$ be the Duflo--Moore operator and consider the weight $\tilde\varphi:=\Tr(D^{1/2}\,\cdot\, D^{1/2})$ on~$\N:=B(L^2(Q))$. Then the action $\Ad\pi$ of $G$ on $\N$ is ergodic, integrable and
$$
\tilde\varphi(T)1=\int_G(\Ad\pi(g))(T)\,dg\quad\text{for}\quad T\in \N_+,
$$
see~\cite{BGNT3}*{Section~2.1}. The corresponding Galois map is defined by
$$
{\mathcal G}\colon L^2(\N,\tilde\varphi)\otimes L^2(\N,\tilde\varphi)\to L^2(G;L^2(\N,\tilde\varphi)),
$$
$$
{\mathcal G}\big(\tilde\Lambda(S)\otimes \tilde\Lambda(T)\big)(g):=\tilde\Lambda(\pi(g)S\pi(g)^*T),
$$
where $\tilde\Lambda\colon{\mathfrak N}_{\tilde\varphi}\to L^2(\N,\tilde\varphi)$ denotes the GNS-map and ${\mathfrak N}_{\tilde\varphi}=\{T:\tilde\varphi(T^*T)<\infty\}$. This map is isometric, and the pair $(\N,\Ad\pi)$ is called a $G$-Galois object if $\mathcal G$ is unitary. Note that for finite~$G$ the map $\G$ is unitary already for dimension reasons.

\begin{theorem}\label{thm:key}
The pair $(B(L^2(Q)),\Ad\pi)$ is a $G$-Galois object.
\end{theorem}

\begin{proof}
Similarly to the proof of \cite{GM}*{Proposition 3.2}, we can describe the Galois map as follows. The GNS-space of the trace $\Tr$ on $B(L^2(Q))$ can be identified with the space $\HS(L^2(Q))$ of Hilbert--Schmidt operators. Hence the same space can be taken for the GNS-space of $\tilde\varphi=\Tr(D^{1/2}\,\cdot\, D^{1/2})$ and then the GNS-map is $\tilde\Lambda(T)=TD^{1/2}$. The space $\HS(L^2(Q))$, in turn, can be identified with $L^2(Q\times Q)$. Note that by Proposition~\ref{prop:representation}, if $T\in B(L^2(Q))$  is a kernel operator with kernel $k\in C_c(Q\times Q)$, then $TD^{1/2}$ is also a kernel operator with kernel
$$
k\big(1\otimes \Delta^{-1/2}|_Q\big)\in L^2(Q\times Q).
$$
We can therefore take for the GNS-space of $\tilde\vf$ the Hilbert space
$$
L^2(Q\times Q,dq\times d\nu),\quad\mbox{where}\quad d\nu(q)=\Delta(q)^{-1}\,dq ,
$$
and then the GNS-map $\tilde\Lambda$ is just the map that associates to a kernel operator its kernel.

Observe next that the operator $\Ad\pi(g)$ on $\HS(L^2(Q))$ corresponds to the operator $\pi(g)\otimes\pi^c(g)$ on $L^2(Q\times Q)$, where $\pi^c(g)\vf=\overline{\pi(g)\bar\vf}$. Hence the Galois map is a map
\begin{align*}
{\mathcal G}: L^2(Q\times Q,dq\times d\nu)\otimes L^2(Q\times Q,dq\times d\nu)\to
L^2\big(G,L^2(Q\times Q,dq\times d\nu)\big)
\end{align*}
such that for $k_1,k_2\in C_c(Q\times Q)$ we have
\begin{align*}
{\mathcal G}(k_1\otimes k_2)(g;q_1,q_2)&=\int\big((\pi(g)\otimes\pi^c(g))k_1\big)(q_1,q_0)\,k_2(q_0,q_2)\,
dq_0 \\
&=\int e^{-i \langle\eta(q_1)-\eta(q_0),v\rangle}\,b (q_1^{-1},q)\,\overline{b (q_0^{-1},q)}\, k_1\big(q^{-1}q_1,q^{-1}q_0\big)\,k_2(q_0,q_2)\,
dq_0\\
&=\int e^{i \langle\xi,v\rangle}\,b (q_1^{-1},q)\,\overline{b (\eta^{-1}(\xi+\eta(q_1))^{-1},q)}\, \\
&\qquad k_1\big(q^{-1}q_1,q^{-1}\eta^{-1}(\xi+\eta(q_1))\big)
\,k_2(\eta^{-1}(\xi+\eta(q_1)),q_2)\,
|\eta^{-1}(\xi+\eta(q_1))|\,d\xi.
\end{align*}

From this we see that $\G$ is the composition of three operators, two of which are the inverse Fourier transform in one variable and the multiplication by a $\T$-valued function, which are both unitary, and one is the operator
$$
U_\Xi\colon L^2(Q\times Q\times Q\times Q, dq\times d\nu\times dq\times d\nu)\to L^2(Q\times\hat V\times Q\times Q,\frac{dq}{|q|}\times d\xi\times dq\times d\nu),
$$
$$
(U_\Xi\vf)(q,\xi,q_1,q_2)=|\eta^{-1}(\xi+\eta(q_1))|\,\vf\big(\Xi(q,\xi,q_1),q_2\big),
$$
where $\Xi\colon Q\times\hat V\times Q\to Q\times Q\times Q$ is the almost everywhere defined map
$$
\Xi(q,\xi,q_1)=(q^{-1}q_1,q^{-1}\eta^{-1}(\xi+\eta(q_1)),\eta^{-1}(\xi+\eta(q_1))).
$$
As $\G$ is isometric, the operator $U_\Xi$ must also be isometric, and hence in order to prove that $\G$ is unitary it suffices to show that $\Xi$ is a measure class isomorphism. But this is clear from the assumption of essential bijectivity of $\eta$.
\end{proof}

By \cite{BGNT3}*{Proposition~2.9}, in order to conclude that the Galois object $(B(L^2(Q)),\Ad\pi)$ is defined by a unitary dual $2$-cocycle on $G$, it remains to find a quantization map $\Op\colon L^2(G)\to \HS(L^2(Q))$, that is, a unitary operator that intertwines the left regular representation $\lambda$ of $G$ with $\Ad\pi$. In~\cite{BGNT3} we used the Kohn--Nirenberg quantization to define such a map. A modification of that construction suitable for projective representations has been suggested in~\cite{GM}. In a similar way we get the following result.

\begin{proposition}
\label{UEQ}
For every $f\in C_c(G)$, let $\Op(f)$ be the sesquilinear form on $C_c(Q)$ given by
 $$
\Op(f)[\vf_1 ,\vf_2 ]:= \int_{G} f(g)\, \bigg(\int_Q(\pi(g)^*\vf_1) (q_0) \,|q_0|^{-1/2}\,dq_0\bigg)\overline{( \pi(g)^*\vf_2 )(e)}
\, dg.
$$
Then we have
$$
\Op(f)[\vf_1 ,\vf_2 ]=\int_{Q\times Q}K(f)(q_0,q)\,\vf_1 (q)\,\overline{\vf_2 (q_0)}\,
dq_0\,  dq,
$$
where
\begin{align}
 \label{Komega}
 K(f)(q_0,q)=|q_0q|^{-1/2}\,b (q^{-1}_0,q_0)\,\overline{b (q^{-1},q_0)}\,(\CF_V f)(q_0, \eta(q_0)-\eta(q)).
\end{align}
Consequently, the map $\Op$ extends to a unitary operator $L^2(G)\to{\HS}(L^2(Q))$ that intertwines the representations $\lambda$ and $\Ad\pi$ of $G$.
\end{proposition}

Here $\CF_Vf$ denotes the partial Fourier transform of $f$ in the variable $v\in V$.

\begin{proof}
From definition \eqref{def-pi} we get that
$$
(\pi(g)^*\vf)(q_0)=e^{i \langle\eta(qq_0),v\rangle}\,\overline{b ((qq_0)^{-1},q)}\,\vf(qq_0).
$$
It follows that
\begin{align*}
\Op(f)[\vf_1 ,\vf_2 ]&=\int e^{i \langle\eta(qq_0)-\eta(q),v\rangle}\, f(q,v)\,\overline{b ((qq_0)^{-1},q)}\,b (q^{-1},q)\\
&\qquad\vf_1(qq_0)\,\overline{\vf_2(q)}\,|q_0|^{-1/2}\,|q|^{-1}\,dq\,dv\,dq_0\\
&=\int e^{i \langle \eta(q_0)-\eta(q),v\rangle}\, f(q,v)\,\overline{b (q_0^{-1},q)}\,b (q^{-1},q)\\
&\qquad\vf_1(q_0)\,\overline{\vf_2(q)}\,|qq_0|^{-1/2}\,dq\,dv\,dq_0,
\end{align*}
which is the required formula for $\Op(f)$ up to swapping $q$ with $q_0$.

Next, by our normalizations of Haar measures, we have
\begin{align*}
\|K(f)\|^2_{L^2(Q\times Q)}&=\int |(\CF_V f)(q_0, \eta(q_0)-\eta(q))|^2\,|q_0q|^{-1}\,dq_0\,dq\\
&=\int |(\CF_V f)(q_0, \eta(q_0)-\xi)|^2\,|q_0|^{-1}\,dq_0\,d\xi\\
&=\int |f(q_0,v)|^2\,|q_0|^{-1}\,dq_0\,dv=\|f\|^2_{L^2(G)}.
\end{align*}
Therefore $K$ extends to an isometry $L^2(G)\to L^2(Q\times Q)$. Similarly to the proof of Theorem~\ref{thm:key}, this operator is actually unitary, since the transformation $Q\times Q\to Q\times\hat V$, $(q_0,q)\mapsto(q_0,\eta(q_0)-\eta(q))$, is a measure class isomorphism. This unitary defines the required extension of $\Op$; to put it differently, $\Op\colon L^2(G)\to \HS(L^2(Q))$ is $K$ under the identification of $L^2(Q\times Q)$ with $\HS(L^2(Q))$.

Finally, to show that $\Op$ intertwines $\lambda$ with $\Ad\pi$ we need to check that $\Op(f(g^{-1}\,\cdot\,))[\vf_1 ,\vf_2 ]=\Op(f)[\pi(g)^*\vf_1 ,\pi(g)^*\vf_2 ]$. But this is immediate by definition.
\end{proof}

\bigskip

\section{Dual \texorpdfstring{$2$}{2}-cocycles and multiplicative unitaries}\label{sec:2mult}

Once we have constructed Galois objects and quantization maps, we can find explicit formulas for the corresponding dual cocycles and the associated multiplicative unitaries. This goes along the lines of~\citelist{\cite{BGNT3}\cite{GM}}, but the computations become more involved. In this section we present the final results and discuss their consequences, but defer the gory details to appendices.

\smallskip

We continue to work in the setting of Section~\ref{sec:rep}. In order to simplify some of the formulas, we will assume that $\beta(e,e)=0$ and hence, for all $q\in Q$, we have
\begin{equation}\label{eq:normalization2}
\beta(e,q)=\beta(q,e)=0.
\end{equation}
This is always possible to achieve by passing to a cohomologous cocycle. By rescaling we can also assume that the cochain $b$ such that $\psi=\partial b$ satisfies
$b(e,e)=1$. Condition \eqref{eq:normalization2} for all $q$ implies that
$$
\beta(q,q^{-1})=q\beta(q^{-1},q).
$$
It also implies that the cocycle $\psi$ is normalized, meaning that $\psi(q_1,q_2,q_3)=1$ if $q_i=e$ for some~$i$. Condition $\psi(q_1,e,q_3)=1$ together with~$b(e,e)=1$ imply, in turn, that
\begin{equation}\label{eq:normalization3}
b(e,q)=b(q,e)=1.
\end{equation}

We know by \cite{BGNT3}*{Proposition 2.9} that there is a dual unitary $2$-cocycle on $G$ defined by the formula
\begin{align}
\label{DBTF}
\Omega:=(\CJ\otimes\CJ)\tilde\G^*(1\otimes\CJ)\hat W,
\end{align}
where we recall that the unitary $\CJ=J\hat J\colon L^2(G)\to L^2(G)$ is given by
$$
(\CJ f)(g)=\Delta(g)^{-1/2}f(g^{-1}),
$$
$\hat W\colon L^2(G\times G)\to L^2(G\times G)$ is the multiplicative unitary of the dual quantum group $\hat G$,
$$
(\hat W f)(g,h)=f(hg,h),
$$
and  $\tilde\G$ is the Galois map  of the $G$-Galois object $(B(L^2(Q)),\Ad \pi )$, transported to $L^2(G\times G)$ by using the quantization map $\Op$.

\begin{theorem}
\label{OmegaF}
In the setting of Section~\ref{sec:rep}, consider the dual unitary $2$-cocycle $\Omega$ on $G$ defined by~\eqref{DBTF}. Then, assuming that the normalization conditions \eqref{eq:normalization2}--\eqref{eq:normalization3} are satisfied, for $f\in L^2(Q\times\hat V\times Q\times\hat V,|q|^{-1}dq\times d\xi\times |q|^{-1}dq\times d\xi)$, we have
\begin{multline}\label{eq:Omega}
\big((\CF_V\otimes\CF_V)\Omega (\CF_V^*\otimes\CF_V^*)f\big)(q_1,\xi_1;q_2,\xi_2) = |\eta^{-1}(\xi_1)|^{-1}\,
\overline {b(\eta^{-1}(\xi_2)^{-1},\eta^{-1}(\xi_1)^{-1})}\\
\overline{\psi(\eta^{-1}(\xi_2)^{-1},\eta^{-1}(\xi_1)^{-1},\eta^{-1}(\xi_1)q_2)}\,f(q_1, \xi_1 ;\eta^{-1}(\xi_1)q_2,\eta^{-1}(\xi_1)^\flat\xi_2).
\end{multline}
\end{theorem}

A proof of this theorem is given in Appendix~\ref{ap:A}.

\smallskip

Note that if $\beta$ and $\eta$ are fixed, we still have some freedom in choosing the cochain $b$, namely, we can multiply it by any Borel $2$-cocycle $c\in Z^2(Q;\T)$. Then the $2$-cocycle $\omega$ defined by~\eqref{eq:omega} gets multiplied by the inflation of~$c$ to~$G$. Let $\Omega_c$ be the dual $2$-cocycle that we get when we replace $b$ by $bc$. By \cite{BGNT3}*{Corollary 2.8} we obtain the following result.

\begin{corollary}
The map $ c \mapsto \Omega_c$ yields an embedding  $H^2(Q;\mathbb T)/{\rm Ker}({\rm Inf}_Q^G)\hookrightarrow H^2(\hat G;\mathbb T)$, where ${\rm Inf}_Q^G\colon H^2(Q;\mathbb T)\to H^2(G;\mathbb T)$ is the  inflation homomorphism and where $H^2(\hat G;\mathbb T)$ denotes the set of equivalence classes of dual unitary $2$-cocycles for $G$.
\end{corollary}

Our next goal is to describe the multiplicative unitary of the quantum group $\hat G_\Omega$ such that its function algebra coincides with $W^*(G)$, but the coproduct is defined by $\Dhat_\Omega:=\Omega\Dhat(\cdot)\Omega^*$. For this recall that a pentagonal transformation is a
measure class isomorphism $w\colon X\times X \to X\times X$, for a standard measure space $X$, satisfying
the relation $w_{23}\circ w_{13}\circ w_{12} = w_{12}\circ w_{23}$. Given such a transformation, we get a unitary $U$ on $L^2(X\times X)$,
$$
(Uf)(x,y)=J_w(x,y)^{1/2}f(w(x,y)),
$$
where $J_w$ is a Radon--Nikodym derivative. It satisfies the pentagon relation
$$
U_{12}U_{13}U_{23} =U_{23}U_{12}.
$$
A measurable function $\theta\colon X\times X\to\T$ is called a pentagonal $2$-cocycle (for $w$), if the unitary~$\theta U$ still satisfies the pentagon relation, where by $\theta$ we mean the operator of multiplication by the function $\theta$.

\begin{theorem}\label{thm:mult}
In the setting of Section~\ref{sec:rep}, consider the dual unitary $2$-cocycle $\Omega$ on $G$ defined by~\eqref{DBTF}.  Then, assuming the normalization conditions \eqref{eq:normalization2}--\eqref{eq:normalization3}, the multiplicative unitary~$\hat W_\Omega$ of the deformed quantum group $(W^*(G),\Omega\Dhat(\cdot)\Omega^*)$ is described as follows: for $f\in L^2(Q\times\hat V\times Q\times\hat V,|q|^{-1}dq\times d\xi\times |q|^{-1}dq\times d\xi)$, we have
\begin{equation}\label{eq:mult}
(\CF_V\otimes\CF_V)\hat W_\Omega (\CF_V^*\otimes\CF_V^*)f=
 \Theta \,J_{ v }^{1/2}\, (f\circ  v ),
\end{equation}
where
$ v \colon (Q\times\hat V)\times (Q\times\hat V)\to (Q\times\hat V)\times (Q\times\hat V)$ is the pentagonal transformation defined by
$$
 v (q_1,\xi_1;q_2,\xi_2):=\big(q_2q_1,q_2^\flat\xi_1;\eta^{-1}(\eta(q_2^{-1})+\xi_1)^{-1}\eta^{-1}(\xi_1),{\eta^{-1}(\eta(q_2^{-1})+\xi_1)^{-1}}^\flat
({q_2^{-1}}^\flat\xi_2-\xi_1)\big),
$$
$J_{v}$ is the operator of multiplication by the function
$$
J_{v}(q_1,\xi_1;q_2,\xi_2):=|\eta^{-1}(\eta(q_2^{-1})+\xi_1)|^2,
$$
and $\Theta$ is the operator of multiplication by the pentagonal $2$-cocycle $(Q\times\hat V)\times (Q\times\hat V)\to\T$ given by
\begin{align}
 \Theta (q_1,\xi_1;q_2,\xi_2):=\ &
 \overline {b(\eta^{-1}({q_2^{-1}}^\flat\xi_2)^{-1}\eta^{-1}(\xi_1),\eta^{-1}(\xi_1)^{-1})}\,
b(\eta^{-1}(\xi_2)^{-1}\eta^{-1}(q_2^\flat\xi_1),\eta^{-1}(q_2^\flat\xi_1)^{-1})  \nonumber\\
 &\psi(\eta^{-1}(q_2^\flat\xi_1)^{-1},q_2,q_1)
\, \overline{\psi(\eta^{-1}(\xi_2- q_2^\flat\xi_1)^{-1},q_2\eta^{-1}(\xi_1),\eta^{-1}(\xi_1)^{-1})}\nonumber\\
 & \psi(\eta^{-1}(\xi_2- q_2^\flat\xi_1)^{-1},\eta^{-1}(q_2^\flat\xi_1),\eta^{-1}(q_2^\flat\xi_1)^{-1})\nonumber\\
 &\overline{\psi(\eta^{-1}(\xi_2- q_2^\flat\xi_1)^{-1},\eta^{-1}(q_2^\flat\xi_1),\eta^{-1}(q_2^\flat\xi_1)^{-1}q_2\eta^{-1}(\xi_1))}.\label{eq:Theta}
\end{align}
\end{theorem}

A proof of this theorem is given in Appendix~\ref{ap:B}.

\smallskip

Let us draw some consequences of the explicit formula for $\hat W_\Omega$. By a result of Baaj--Skandalis~\cite{MR1652717} (see also \cite{MR1969723}*{Proposition~5.1} for a correction), under mild assumptions a pentagonal transformation $w\colon X\times X\to X\times X$ is defined by a matched pair of groups, that is, $X$ can be identified with a locally compact group $L$ having a matched pair of subgroups $(L_1,L_2)$ so that we have
$$
w(x,y)=(xp_1(p_2(x)^{-1}y),p_2(x)^{-1}y)
$$
almost everywhere, where $p_i\colon L_1L_2\cong L_1\times L_2\to L_i$ is the projection onto the $i$-th factor. The corresponding unitary is the multiplicative unitary of the bicrossed product $\hat L_1\bicrossl L_2$, which is a locally compact quantum group with function algebra $L^\infty(\hat L_1\bicrossl L_2)=L_1\ltimes L^\infty(L_2)$.

\begin{corollary}
The locally compact quantum group $(W^*(G),\Omega\hat\Delta(.)\Omega^*)$
is isomorphic to the cocycle bicrossed product quantum group defined
by the matched pair $(Q,Q_\eta)$ of subgroups of $Q\ltimes\hat V$ and the pentagonal $2$-cocycle $ \Theta\circ(T^{-1}\times T^{-1}) $, where $Q_\eta=\{(q,\eta(q)):q\in Q\}$ and $T\colon Q\ltimes\hat V\to Q\ltimes \hat V$ is the measure isomorphism defined by
$$
T(q,\xi):=(q^{-1}\eta^{-1}(\xi),{q^{-1}}^\flat\xi).
$$
\end{corollary}

We record the formula for $T^{-1}$ for later use:
\begin{equation}\label{eq:T-inverse}
T^{-1}(q,\xi):=(\eta^{-1}(\eta(q)-\xi)^{-1},{\eta^{-1}(\eta(q)-\xi)^{-1}}^\flat\xi).
\end{equation}

\bp
This is similar to analogous statements in \cites{BGNT3,GM}, so we will just stress the main points.

First of all recall that $(Q,Q_\eta)$ is a matched pair by Lemma~\ref{lem:matched}. The fact that the unitary on the right hand side of~\eqref{eq:mult} satisfies the pentagon relation is equivalent to saying that $v$ is a pentagonal transformation of the measurable space $(Q\times\hat V)\times(Q\times\hat V)$ and $\Theta$ is a pentagonal $2$-cocycle for~$v$. The transformation $v$ reduces to that obtained in~\cite{BGNT3} when $\eta(q)=q^\flat\xi_0-\xi_0$. If we replace the expressions $q^\flat\xi_0-\xi_0$ and $\phi^{-1}(\xi_0+\xi)$  in~\cite{BGNT3} by $\eta(q)$ and $\eta^{-1}(\xi)$, resp., then the formulas in the proof  of~\cite{BGNT3}*{Theorem~4.1} remain true for general essentially bijective cocycles. For example, the (almost everywhere defined) inverse of $v$ is given by
\begin{align*}
v^{-1}(q_1,\xi_1;q_2,\xi_2)&=\big(\eta^{-1}(\eta^{-1}(\xi_1)^\flat\eta(q_2))^{-1}q_1,{\eta^{-1}(\eta^{-1}(\xi_1)^\flat\eta(q_2))^{-1}}^\flat\xi_1;\\
&\qquad\qquad\eta^{-1}(\eta^{-1}(\xi_1)^\flat\eta(q_2)),\xi_1+\eta^{-1}(\xi_1)^\flat\xi_2\big).
\end{align*}
Then the same considerations as in that proof show that if we define $f_1\colon Q\ltimes\hat V\to Q$, $f_2\colon Q\ltimes \hat V\to Q_\eta$ and $T\colon Q\ltimes\hat V\to Q\ltimes\hat V$ by
$$
f_1(q,\xi):=q^{-1},\quad f_2(q,\xi):=(\eta^{-1}(\xi)^{-1},\eta(\eta^{-1}(\xi)^{-1})),\quad T(x):=f_1(x)f_2(x)^{-1},
$$
then $T\times T$ intertwines $v$ with the pentagonal transformation $w$ defined by the matched pair $(Q,Q_\eta)$.

Once we have identified $v$ with the pentagonal transformation defined by $(Q,Q_\eta)$, the pentagonal cocycle $\Theta$ becomes a $2$-cocycle in the measurable Kac cohomology of this pair~\cite{MR2115374} and hence, by results of~\cite{MR1970242}, $\hat W_\Omega$ is the multiplicative unitary of a cocycle bicrossed product as stated in the formulation.
\ep

A pentagonal $2$-cocycle $\theta$ for a pentagonal transformation $w$ is called a coboundary if it has the form
$$
\theta=\frac{f\otimes f}{(f\otimes f)\circ w}
$$
for a $\T$-valued measurable function $f$. We denote by $H^2_\pent(w;\T)$ the quotient of the group of pentagonal $2$-cocycles (identified almost everywhere) by the subgroup of coboundaries. When the pentagonal transformation $w$ is defined by a matched pair $(L_1,L_2)$, we denote this group by $H^2_\pent(L_1\bowtie L_2;\T)$.

\smallskip

By multiplying $b$ in Theorem~\ref{thm:mult} by $2$-cocycles $c$ on $Q$ we now get the following result.

\begin{corollary}\label{cor:pent}
We have a homomorphism $H^2(Q;\T)\to H^2_\pent(v;\T)$ defined by $[c]\mapsto[\theta_c]$,
\begin{equation}\label{eq:theta-c}
\theta_c(q_1,\xi_1;q_2,\xi_2):= \overline {c(\eta^{-1}({q_2^{-1}}^\flat\xi_2)^{-1}\eta^{-1}(\xi_1),\eta^{-1}(\xi_1)^{-1})}\,
c(\eta^{-1}(\xi_2)^{-1}\eta^{-1}(q_2^\flat\xi_1),\eta^{-1}(q_2^\flat\xi_1)^{-1}).
\end{equation}
\end{corollary}

Since $H^2(v;\T)\cong H^2(Q\bowtie Q_\eta;\T)$, we thus get a homomorphism $H^2(Q;\T)\to H^2(Q\bowtie Q_\eta;\T)$.
We next want to compare this homomorphism with homomorphisms in the Kac exact sequence~\citelist{\cite{Kac}\cite{MR2115374}}.

For a matched pair $(L_1,L_2)$ of subgroups of $L$, the part of the Kac sequence that is important for us has the form
$$
H^2(L;\T)\to H^2(L_1;\T)\oplus H^2(L_2;\T)\to H^2_\pent(L_1\bowtie L_2;\T)\to H^3(L;\T).
$$
Here the homomorphisms $H^2(L;\T)\to H^2(L_i;\T)$ are simply the restriction maps. In order to describe the homomorphisms $H^2(L_i;\T)\to H^2(L_1\bowtie L_2;\T)$, let us denote by $\alpha$ and $\beta$ the left actions of $L_1$ on $L_2$ and of $L_2$ on $L_1$, respectively, defined by the identity
$$
gh^{-1}=\alpha_{g}(h)^{-1}\beta_h(g)\quad\text{for}\quad g\in L_1,\ h\in L_2,
$$
or in other words, by the identities $\alpha_{g}(h)=p_2(hg^{-1})$, $\beta_h(g)=p_1(hg^{-1})^{-1}$.
The following is checked by a straightforward but tedious application of definitions and isomorphisms in~\cite{MR2115374}.

\begin{lemma}
The maps $H^2(L_i;\T)\to H^2_\pent(L_1\bowtie L_2;\T)$, $[c]\mapsto[\kappa_c]$, in the Kac exact sequence are given by
\begin{align*}
\kappa_c(x,y)&=c(x_1,y_1)\,\overline{c(\beta_{x_2^{-1}}(x_1),\beta_{x_2^{-1}}(y_1^{-1})^{-1})} &\text{for}\quad c\in Z^2(L_1;\T),\\
\kappa_c(x,y)&=c(x_2,x_2^{-1}\alpha_{y_1}(y_2^{-1})^{-1})\,\overline{c(\alpha_{y_1^{-1}}(x_2^{-1})^{-1},\alpha_{y_1^{-1}}(x_2^{-1})y_2)}&\text{for}\quad c\in Z^2(L_2;\T),
\end{align*}
where $x_i:=p_i(x)$, $y_i:=p_i(y)$ for $x,y\in L$.
\end{lemma}

We apply this to $L=Q\ltimes\hat V$, $L_1=Q$ and $L_2=Q_\eta$. To simplify the formulas we will view $p_i\colon L\to L_i$ as maps $Q\ltimes\hat V\to Q$. With this convention, for almost all $x\in Q\ltimes\hat V$ we have
$$
x=(x_1,0)(x_2,\eta(x_2))=(x_1x_2,x_1^\flat\eta(x_2)).
$$
Formula~\eqref{eq:T-inverse} then becomes
$$
T^{-1}(x)=(x_1^{-1},\eta(x_2)).
$$
Therefore for the cocycle~\eqref{eq:theta-c} we get
$$
\theta_c(T^{-1}(x),T^{-1}(y))=
\overline{c(\eta^{-1}(y_1^\flat\eta(y_2))^{-1}x_2,x_2^{-1})}\,
c(y_2^{-1}\eta^{-1}({y_1^{-1}}^\flat\eta(x_2)),\eta^{-1}({y_1^{-1}}^\flat\eta(x_2))^{-1}).
$$

Next, it is easy to check that the actions $\alpha,\beta$ of $Q$ on itself defined by the matched pair $(Q,Q_\eta)$ are given by
$$
\alpha_q(q')=\eta^{-1}(q^\flat\eta({q'}^{-1}))^{-1},\qquad\beta_q=\alpha_q.
$$
We thus get
$$
\theta_c(T^{-1}(x),T^{-1}(y))=
\overline{c(\alpha_{y_1}(y_2^{-1})x_2,x_2^{-1})}\,
c(y_2^{-1}\alpha_{y_1^{-1}}(x_2^{-1})^{-1},\alpha_{y_1^{-1}}(x_2^{-1})).
$$
From this we see that
$$
\theta_c\circ(T^{-1}\times T^{-1})=\kappa_{\tilde c},
$$
where $\tilde c\in Z^2(Q_\eta;\T)\cong Z^2(Q;\T)$ is the cocycle given by
$$
\tilde c(q,q'):=\overline{c({q'}^{-1},q^{-1})}.
$$
The cocycles $\tilde c$ and $c$ are cohomologous, they even coincide when $c$ satisfies the normalization condition $c(q,q^{-1})=1$ for all $q$.

To summarize, we have shown that up to the isomorphisms $H^2(Q;\T)\cong H^2(Q_\eta;\T)$ and $H^2_\pent(v;\T)\cong H^2_\pent(Q\bowtie Q_\eta;\T)$, the homomorphism $H^2(Q;\T)\to H^2_\pent(v;\T)$ defined by Corollary~\ref{cor:pent} coincides with the homomorphism $H^2(Q_\eta;\T)\to H^2_\pent(Q\bowtie Q_\eta;\T)$ in the Kac exact sequence
$$
H^2(Q\ltimes\hat V;\T)\to H^2(Q;\T)\oplus H^2(Q_\eta;\T)\to H^2_\pent(Q\bowtie Q_\eta;\T)\to H^3(Q\ltimes\hat V;\T).
$$

As an application we can now strengthen \cite{GM}*{Proposition~3.21}, proving a conjecture made in~\cite{GM}.

\begin{proposition}\label{prop:inj}
If our fixed essentially bijective cocycle $\eta\in Z^1_c(Q;\hat V)$ is a coboundary, then the homomorphism $H^2(Q;\T)\to H^2_\pent(v;\T)$ defined by Corollary~\ref{cor:pent} is injective.
\end{proposition}

Therefore if in the setting of Theorem~\ref{thm:mult} we denote by $\Theta_c$ the pentagonal cocycle defined by~\eqref{eq:Theta}, with $b$ replaced by $bc$, then assuming that $\eta$ is a coboundary, we conclude that the map $H^2(Q;\T)\to H^2_\pent(v;\T)$, $[c]\mapsto[\Theta_c]$, is injective.

\bp[Proof of Proposition~\ref{prop:inj}]
The Kac exact sequence implies that injectivity of the homomorphism $H^2(Q;\T)\to H^2_\pent(v;\T)$ is equivalent to the following property: if the restriction of a $2$-cocycle on $Q\ltimes\hat V$ to $Q$ is a coboundary, then its restriction to~$Q_\eta$ is a coboundary as well. But this property is obviously true when $\eta$ is a coboundary, since then~$Q$ and~$Q_\eta$ are conjugate subgroups of $Q\ltimes\hat V$.
\ep

When $\eta$ is not a coboundary, the homomorphism $H^2(Q;\T)\to H^2_\pent(v;\T)$ is not injective in general. Equivalently, there might be nontrivial cocycles on $Q_\eta$ that can be extended to cocycles on $Q\ltimes\hat V$ such that their restrictions to $Q\subset Q\ltimes\hat V$ are coboundaries. For example, given $\gamma\in Z^1(Q;V)$ we can consider the cup-product
$$
(\gamma\cup\eta)(q,q')=e^{i\langle\gamma(q),q^\flat\eta(q')\rangle},
$$
the required extension of $\gamma\cup\eta$ from $Q_\eta$ to $Q\ltimes\hat V$ is given by
$$
\nu_\gamma((q_1,\xi_1),(q_2,\xi_2)):=e^{i\langle\gamma(q_1),q_1^\flat\xi_2\rangle}.
$$

The last cocycle has the extra property that not only is its restriction to $Q$ trivial, but also its restriction to $\hat V$ is trivial. More generally, one can show that if we have a $2$-cocycle $\nu$ on $Q\ltimes\hat V$ such that its restriction $\Res^{Q\ltimes\hat V}_Q\nu$ to $Q$ is a coboundary, then the cohomology class of $\Res^{Q\ltimes\hat V}_{\hat V}\nu$ is $Q$-invariant and lies in the kernel of a natural homomorphism $\tau\colon H^2(\hat V;\T)^Q\to H^2(Q;V)$ (see \citelist{\cite{MR1810480}*{Section~1.2}\cite{MR2844801}*{Section~3}} for the definition of $\tau$), and this way we get a short exact sequence
\begin{multline*}
0\to H^1(Q;V)\xrightarrow{[\gamma]\mapsto[\nu_\gamma]}\ker(\Res^{Q\ltimes\hat V}_Q\colon H^2(Q\ltimes\hat V;\T)\to  H^2(Q;\T))\\
\xrightarrow{\Res^{Q\ltimes\hat V}_{\hat V}}\ker(\tau\colon H^2(\hat V;\T)^Q\to H^2(Q;V))\to0.
\end{multline*}
The kernel of $H^2(Q;\T)\to H^2_\pent(v;\T)$ is isomorphic to the image of
$$\ker(\Res^{Q\ltimes\hat V}_Q\colon H^2(Q\ltimes\hat V;\T)\to H^2(Q;\T))$$ under the restriction map $\Res^{Q\ltimes\hat V}_{Q_\eta}\colon H^2(Q\ltimes\hat V;\T)\to H^2(Q_\eta;\T)$.

\bigskip

\appendix

\section{Proof of Theorem~\ref{OmegaF}}\label{ap:A}

To compute the dual cocycle $\Omega$ it is convenient to introduce the following function space.

\begin{definition}[cf.~\cite{BGNT3}*{Definition~3.15}]
Denote by $\CF\CL_0(G)$ the space of functions $f\in L^2(G)$ such that $\CF_V f$ is (essentially) bounded on~$Q\times\hat V$ and zero outside the set
$$
X_{L,M}:=\{(q,\xi):q\in L,\ \eta(q)-\xi\in \eta(M)\}
$$
for some compact subsets $L,M\subset Q$.
\end{definition}

This space is dense in $L^2(G)$, since the union of the sets $X_{L,M}$ is a subset of $Q\times\hat V$ of full measure by essential bijectivity of $\eta$.

From the cocycle relation for $\beta\in Z^2(Q;V)$ and our normalization~\eqref{eq:normalization3},
we have
\begin{equation}\label{eq:FL-invariance}
(\CF_V\lambda_{g} f)(q',\xi')=|q| e^{-i\langle \xi',v+\beta(q,q^{-1}q')\rangle}(\CF_Vf)(q^{-1}q',{q^{-1}}^\flat\xi').
\end{equation}
As $q^\flat\eta(M)+\eta(q)=\eta(qM)$ and
$\eta(q')-\xi'=q^\flat(\eta(q^{-1}q')-{q^{-1}}^\flat\xi')+\eta(q)$,
we see that if $\CF_V f$ is zero outside $X_{L,M}$, then $\CF_V\lambda_{g} f$ is zero outside $X_{qL,qM}$. It follows that the space $\CF\CL_0(G)$ is invariant under left translations by the elements of $G$.

\begin{lemma}
\label{Gomega}
For every $f\in L^2(Q\times\hat V\times Q\times\hat V,|q|^{-1}dq\times d\xi\times |q|^{-1}dq\times d\xi)$, we have:

\medskip
$\displaystyle ((\CF_V\otimes \CF_V)\tilde\G(\CF_V^*\otimes \CF_V^*)f)(q_1,\xi_1;q_2,\xi_2)$
\begin{align*}
=\ &
|q_1|\Delta(q_1)^{-1/2}\Delta(\eta^{-1}(\xi_1+\eta(q_2)))^{1/2}\,e^{i\langle\xi_1,\beta(q_1,q_1^{-1}q_2)\rangle} \\
& b (\eta^{-1}(\eta(q_2)-\xi_2)^{-1},q_2)\,
\overline{b (\eta^{-1}(\xi_1+\eta(q_2))^{-1},q_2)}\\
& b (\eta^{-1}(\xi_1+\eta(q_2))^{-1},\eta^{-1}(\xi_1+\eta(q_2)))\,\overline{b(\eta^{-1}(\eta(q_2)-\xi_2)^{-1},\eta^{-1}(\xi_1+\eta(q_2)))}
\\
& f(q_1^{-1}q_2,-{q_1^{-1}}^\flat\xi_1;
\eta^{-1}(\xi_1+\eta(q_2)), \xi_1+\xi_2).
\end{align*}
\end{lemma}

\begin{proof}
If $f\in \CF\CL_0(G)$ and $\CF_Vf$ is zero outside $X_{L,M}$, then the function $K(f)$ given by~\eqref{Komega} is zero outside $L\times M$. It follows that if $D$ is the Duflo--Moore operator from Proposition~\ref{prop:representation}, then $\Op(f)D^{-1/2}$ is a kernel operator with kernel $K(f)(1\otimes
\Delta^{1/2}|_Q)$, which is a bounded function vanishing outside $L\times M$. In particular, $\Op(f)D^{-1/2}$ is  Hilbert--Schmidt and thus bounded.

Take $f_1,f_2\in \CF\CL_0(G)$. Then for the Galois map $\tilde \G$ we have
\begin{align*}
(\tilde \G(f_1\otimes f_2))(g_1,g_2)&=\Op^*\big(\pi (g_1)\Op(f_1)D^{-1/2}\pi (g_1)^*
\Op(f_2)\big)(g_2)\\
&=\Delta(g_1)^{-1/2}\Op^*\big(\Op(\lambda_{g_1}f_1)D^{-1/2}
\Op(f_2)\big)(g_2).
\end{align*}
The operator $\Op(\lambda_{g_1}f_1)D^{-1/2}\Op(f_2)$ has kernel $ k_{g_1}\in L^2(Q\times Q)$,
given by
\begin{align*}
k_{g_1}(q_2,q_3)
&= \int K(\lambda_{g_1}f_1)(q_2,q_0)\,\Delta(q_0)^{1/2} \,
 K(f_2)(q_0,q_3)\,dq_0\\
 &= \int b (q_2^{-1},q_2)
\,\overline{b(q_0^{-1},q_2)}\,b (q_0^{-1},q_0)
\,\overline{b(q_3^{-1},q_0)}\,\Delta(q_0)^{1/2}
 \\
& \qquad(\CF_V \lambda_{g_1}f_1)(q_2, \eta(q_2)-\eta(q_0))\,(\CF_V f_2)(q_0, \eta(q_0)-\eta(q_3))
 \,|q_2q_3|^{-1/2}\,|q_0|^{-1}\,dq_0.
 \end{align*}

We need to write this as $K(f)$ for some $f$. For this we have to invert the kernel formula~\eqref{Komega}. Letting
$$
k(q_0,q):=|q_0q|^{-1/2}\,b (q_0^{-1},q_0)
\,\overline{b (q^{-1},q_0)}\,(\CF_V f)(q_0, \eta(q_0)-\eta(q))
$$
and $\xi= \eta(q_0)-\eta(q)$, so that $q=\eta^{-1}(\eta(q_0)-\xi)$, we get
\begin{multline*}
(\CF_V f)(q_0, \xi)\\=|q_0\eta^{-1}(\eta(q_0)-\xi)|^{-1/2}\,\overline{b (q_0^{-1},q_0)}
\,{b }(\eta^{-1}(\eta(q_0)-\xi)^{-1},q_0)\,k(q_0,\eta^{-1}(\eta(q_0)-\xi)).
\end{multline*}

We therefore deduce:

\medskip
$\displaystyle \big((1\otimes \CF_V)\tilde \G(f_1\otimes f_2)\big)(q_1,v_1;q_2,\xi_2)$
\begin{align*}
&=\Delta(q_1)^{-1/2}\,
|q_2\eta^{-1}(\eta(q_2)-\xi_2)|^{1/2}\,\overline{b(q_2^{-1},q_ 2)}\,b (\eta^{-1}(\eta(q_2)-\xi_2)^{-1},q_2)\\
&\qquad
k_{g_1}(q_2,\eta^{-1}(\eta(q_2)-\xi_2))\\
&=\Delta(q_1)^{-1/2}
\, b (\eta^{-1}(\eta(q_2)-\xi_2)^{-1},q_2)\int \overline{b(q_0^{-1},q_2)}\,b (q_0^{-1},q_0)
\,\overline{b(\eta^{-1}(\eta(q_2)-\xi_2)^{-1},q_0)} \\
& \qquad(\CF_V \lambda_{g_1}f_1)(q_2, \eta(q_2)-\eta(q_0))\,(\CF_V f_2)(q_0, \eta(q_0)-\eta(q_2)+\xi_2)
 \,\Delta(q_0)^{1/2}\,|q_0|^{-1}\,dq_0.
\end{align*}
Applying~\eqref{eq:FL-invariance} we then get

\medskip
$\displaystyle \big((1\otimes \CF_V)\tilde \G(f_1\otimes f_2)\big)(q_1,v_1;q_2,\xi_2)$
\begin{align}\label{eq:star-product0}
=\ &|q_1|\,\Delta(q_1)^{-1/2}
\, b (\eta^{-1}(\eta(q_2)-\xi_2)^{-1},q_2)\notag\\
&\int \overline{b (q_0^{-1},q_2)}\,b (q_0^{-1},q_0)
\,\overline{b (\eta^{-1}(\eta(q_2)-\xi_2)^{-1},q_0)}\notag \\
&(\CF_Vf_1)(q_1^{-1}q_2,{q_1^{-1}}^\flat(\eta(q_2)-\eta(q_0)))\,(\CF_V f_2)(q_0, \eta(q_0)-\eta(q_2)+\xi_2)\notag\\
&e^{-i\langle \eta(q_2)-\eta(q_0),v_1+\beta(q_1,q_1^{-1}q_2)\rangle}\,\Delta(q_0)^{1/2}\,\frac{dq_0}{\,|q_0|}.
\end{align}
Setting $\xi_1:=\eta(q_0)$, we get

\medskip
$\displaystyle \big((1\otimes \CF_V)\tilde \G(f_1\otimes f_2)\big)(q_1,v_1;q_2,\xi_2)$
\begin{align*}
=\ &|q_1|\,\Delta(q_1)^{-1/2}
\, b (\eta^{-1}(\eta(q_2)-\xi_2)^{-1},q_2)\\
&\int \overline{b (\eta^{-1}(\xi_1)^{-1},q_2)}\,b (\eta^{-1}(\xi_1)^{-1},\eta^{-1}(\xi_1))
\,\overline{b (\eta^{-1}(\eta(q_2)-\xi_2)^{-1},\eta^{-1}(\xi_1))}\\
& (\CF_Vf_1)(q_1^{-1}q_2,{q_1^{-1}}^\flat(\eta(q_2)-\xi_1))\,(\CF_V f_2)(\eta^{-1}(\xi_1),\xi_1-\eta(q_2)+\xi_2)\\
& e^{-i\langle \eta(q_2)-\xi_1,v_1+\beta(q_1,q_1^{-1}q_2)\rangle}\,\Delta(\eta^{-1}(\xi_1))^{1/2}\,d\xi_1.
\end{align*}
Performing the translation $\xi_1\mapsto \xi_1+\eta(q_2)$ and using ${q^{-1}}^\flat\eta^{-1}(\xi+\eta(q))=\eta^{-1}( {q^{-1}}^\flat\xi)$, we get

\medskip
$\displaystyle \big((1\otimes \CF_V)\tilde \G(f_1\otimes f_2)\big)(q_1,v_1;q_2,\xi_2)$
\begin{align*}
=\ &|q_1|\,\Delta(q_1)^{-1/2}
\, b (\eta^{-1}(\eta(q_2)-\xi_2)^{-1},q_2)\\
&\int \overline{b (\eta^{-1}(\xi_1+\eta(q_2))^{-1},q_2)}\,b (\eta^{-1}(\xi_1+\eta(q_2))^{-1},\eta^{-1}(\xi_1+\eta(q_2)))\\
&\overline{b (\eta^{-1}(\eta(q_2)-\xi_2)^{-1},\eta^{-1}(\xi_1+\eta(q_2)))}\\
& (\CF_Vf_1)(q_1^{-1}q_2,-{q_1^{-1}}^\flat\xi_1)\,(\CF_V f_2)(\eta^{-1}(\xi_1+\eta(q_2)),\xi_1+\xi_2)\\
& e^{i\langle\xi_1,v_1+\beta(q_1,q_1^{-1}q_2)\rangle}\,\Delta(\eta^{-1}(\xi_1+\eta(q_2)))^{1/2}\,d\xi_1.
\end{align*}

Applying $\CF_V\otimes1$ we conclude that the lemma is true for the functions $f=\CF_Vf_1\otimes \CF_Vf_2$, with $f_1,f_2\in \CF\CL_0(G)$, which is enough by density of the linear span of such functions.
\end{proof}

\begin{lemma}
\label{LDF}
We have almost everywhere:
\begin{multline}\label{eq:LDF}
 \overline{b (\eta^{-1}(\xi_2)^{-1}q,q^{-1}\eta^{-1}(\xi_1)^{-1})}\,b (\eta^{-1}(\xi_2)^{-1}q,q^{-1})\, \overline{b (q,q^{-1})}\,b (q,q^{-1}\eta^{-1}(\xi_1)^{-1})\\
 e^{i\langle {q^{-1}}^\flat\xi_2,\beta(q^{-1},q)-\beta(q^{-1}\eta^{-1}(\xi_1)^{-1},\eta^{-1}(\xi_1)q)\rangle}
 =\overline {b(\eta^{-1}(\xi_2)^{-1},\eta^{-1}(\xi_1)^{-1})}\\ \overline{\psi(\eta^{-1}(\xi_2)^{-1},\eta^{-1}(\xi_1)^{-1},\eta^{-1}(\xi_1)q)}.
 \end{multline}
\end{lemma}
\begin{proof}
We first write the left hand side of~\eqref{eq:LDF} as follows:
\begin{multline*}
\frac{b (\eta^{-1}(\xi_2)^{-1}q, q^{-1}\eta^{-1}(\xi_2))}{b (\eta^{-1}(\xi_2)^{-1}q,q^{-1}\eta^{-1}(\xi_1)^{-1})}
\frac{b (\eta^{-1}(\xi_2)^{-1}q,q^{-1})}{b (\eta^{-1}(\xi_2)^{-1}q, q^{-1}\eta^{-1}(\xi_2))}
\frac{b (q,q^{-1}\eta^{-1}(\xi_1)^{-1})}{b(q,q^{-1})}\\
\psi(\eta^{-1}({q^{-1}}^\flat\xi_2)^{-1},q^{-1},q)\,
\overline  {\psi(\eta^{-1}({q^{-1}}^\flat\xi_2)^{-1},q^{-1}\eta^{-1}(\xi_1)^{-1},\eta^{-1}(\xi_1)q)}.
\end{multline*}
Using $\psi=\partial b$ and \eqref{eq:normalization3}, we have
\begin{align}
\label{ID}
\frac{b(q,q^{-1}q')}{b(q,q^{-1})}=\frac{\psi(q,q^{-1},q')}{b(q^{-1},q')}.
\end{align}
From this we see that the left hand side of~\eqref{eq:LDF} equals
\begin{multline*}
\frac{b(q^{-1}\eta^{-1}(\xi_2),\eta^{-1}(\xi_2)^{-1}\eta^{-1}(\xi_1)^{-1})}{\psi(\eta^{-1}(\xi_2)^{-1}q,q^{-1}\eta^{-1}(\xi_2),\eta^{-1}(\xi_2)^{-1}\eta^{-1}(\xi_1)^{-1})}
\frac{\psi(\eta^{-1}(\xi_2)^{-1}q,q^{-1}\eta^{-1}(\xi_2),\eta^{-1}(\xi_2)^{-1})}{ b(q^{-1}\eta^{-1}(\xi_2),\eta^{-1}(\xi_2)^{-1})}\\
\frac{\psi(q,q^{-1},\eta^{-1}(\xi_1)^{-1})}{ b (q^{-1},\eta^{-1}(\xi_1)^{-1})} \psi(\eta^{-1}({q^{-1}}^\flat\xi_2)^{-1},q^{-1},q)\,
\overline  {\psi(\eta^{-1}({q^{-1}}^\flat\xi_2)^{-1},q^{-1}\eta^{-1}(\xi_1)^{-1},\eta^{-1}(\xi_1)q)}.
\end{multline*}
Using again $\psi=\partial b$, we have the following identity:
\begin{multline*}
\frac{b(q^{-1}\eta^{-1}(\xi_2),\eta^{-1}(\xi_2)^{-1}\eta^{-1}(\xi_1)^{-1})}{ b(q^{-1}\eta^{-1}(\xi_2),\eta^{-1}(\xi_2)^{-1})\,b (q^{-1},\eta^{-1}(\xi_1)^{-1})}=\\
\overline {b(\eta^{-1}(\xi_2)^{-1},\eta^{-1}(\xi_1)^{-1})}\,\psi(q^{-1}\eta^{-1}(\xi_2),\eta^{-1}(\xi_2)^{-1},\eta^{-1}(\xi_1)^{-1}).
\end{multline*}
Applying this we arrive at the following expression for the left hand side of~\eqref{eq:LDF}:
\begin{multline*}
\overline {b(\eta^{-1}(\xi_2)^{-1},\eta^{-1}(\xi_1)^{-1})}\,
\overline {\psi(\eta^{-1}(\xi_2)^{-1}q,q^{-1}\eta^{-1}(\xi_2),\eta^{-1}(\xi_2)^{-1}\eta^{-1}(\xi_1)^{-1})}\\
\psi(\eta^{-1}(\xi_2)^{-1}q,q^{-1}\eta^{-1}(\xi_2),\eta^{-1}(\xi_2)^{-1})\,\psi(q,q^{-1},\eta^{-1}(\xi_1)^{-1})\\
\psi(q^{-1}\eta^{-1}(\xi_2),\eta^{-1}(\xi_2)^{-1},\eta^{-1}(\xi_1)^{-1})\,
\psi(\eta^{-1}({q^{-1}}^\flat\xi_2)^{-1},q^{-1},q)\\
\overline  {\psi(\eta^{-1}({q^{-1}}^\flat\xi_2)^{-1},q^{-1}\eta^{-1}(\xi_1)^{-1},\eta^{-1}(\xi_1)q)}.
\end{multline*}

To prove the lemma it remains to show that the product of the six factors involving $\psi$ in the above expression equals
$$
\overline{\psi(\eta^{-1}(\xi_2)^{-1},\eta^{-1}(\xi_1)^{-1},\eta^{-1}(\xi_1)q)}=e^{-\langle\xi_2,\beta(\eta^{-1}(\xi_1)^{-1},\eta^{-1}(\xi_1)q)\rangle}.
$$
By definition of $\psi$, the product of these six factors is equal to
\begin{multline*}
e^{i\langle\eta(q^{-1}\eta^{-1}(\xi_2)),\beta(q^{-1}\eta^{-1}(\xi_2),\eta^{-1}(\xi_2)^{-1})\rangle}\,
e^{i\langle\eta( \eta^{-1}(\xi_2)^{-1}q),\beta(\eta^{-1}(\xi_2)^{-1},\eta^{-1}(\xi_1)^{-1})\rangle}\\
e^{i\langle\eta(q^{-1}),\beta(q^{-1},\eta^{-1}(\xi_1)^{-1})\rangle}
e^{-i\langle\eta(q^{-1}\eta^{-1}(\xi_2)),\beta(q^{-1}\eta^{-1}(\xi_2),\eta^{-1}(\xi_2)^{-1}\eta^{-1}(\xi_1)^{-1})\rangle}\\
e^{i\langle\xi_2,q\beta(q^{-1},q)\rangle}\,
e^{-i\langle\xi_2,q\beta(q^{-1}\eta^{-1}(\xi_1)^{-1},\eta^{-1}(\xi_1)q)\rangle}.
\end{multline*}
Using the cocycle identity for $\eta$ and its consequence $\eta(q^{-1})=-{q^{-1}}^\flat\eta(q)$, we see that the expression above can be written as
$$
e^{i\langle\xi_2,A(q,\xi_1,\xi_2)\rangle}\,e^{i\langle\eta(q^{-1}),B(q,\xi_1,\xi_2)\rangle},
$$
where
\begin{multline*}
A(q,\xi_1,\xi_2)=q\big(\beta(q^{-1},q)-\beta(q^{-1}\eta^{-1}(\xi_1)^{-1},\eta^{-1}(\xi_1)q)
-\beta(q^{-1}\eta^{-1}(\xi_2),\eta^{-1}(\xi_2)^{-1}\eta^{-1}(\xi_1)^{-1})\\
+\beta(q^{-1}\eta^{-1}(\xi_2),\eta^{-1}(\xi_2)^{-1})- q^{-1}\eta^{-1}(\xi_2)\beta(\eta^{-1}(\xi_2)^{-1},\eta^{-1}(\xi_1)^{-1})\big),
\end{multline*}
and
\begin{multline*}
B(q,\xi_1,\xi_2)=\beta(q^{-1},\eta^{-1}(\xi_1)^{-1})-\beta(q^{-1}\eta^{-1}(\xi_2),\eta^{-1}(\xi_2)^{-1}\eta^{-1}(\xi_1)^{-1})\\
+
\beta(q^{-1}\eta^{-1}(\xi_2),\eta^{-1}(\xi_2)^{-1})-q^{-1}\eta^{-1}(\xi_2)\beta(\eta^{-1}(\xi_2)^{-1},\eta^{-1}(\xi_1)^{-1}).
\end{multline*}
The cocycle identity for $\beta$ at $(q^{-1}\eta^{-1}(\xi_2),\eta^{-1}(\xi_2)^{-1},\eta^{-1}(\xi_1)^{-1})$ gives $B=0$ and
$$
A(q,\xi_1,\xi_2)=q\big(\beta(q^{-1},q)-\beta(q^{-1}\eta^{-1}(\xi_1)^{-1},\eta^{-1}(\xi_1)q)-\beta(q^{-1},\eta^{-1}(\xi_1)^{-1})\big).
$$
Applying the cocycle identity for $\beta$ to the triple $(q^{-1},q,q^{-1}\eta^{-1}(\xi_1)^{-1})$ and then to
the triple $(q,q^{-1}\eta^{-1}(\xi_1)^{-1},\eta^{-1}(\xi_1)q)$, we get
\begin{align*}
A(q,\xi_1,\xi_2)&=\beta(q,q^{-1}\eta^{-1}(\xi_1)^{-1}) -q\beta(q^{-1}\eta^{-1}(\xi_1)^{-1},\eta^{-1}(\xi_1)q)\\
&=-\beta(\eta^{-1}(\xi_1)^{-1},\eta^{-1}(\xi_1)q),
\end{align*}
which completes the proof of the lemma.
\end{proof}

We are now ready to compute the dual cocycle $\Omega$.

\begin{proof}[Proof of Theorem~\ref{OmegaF}]
We shall first compute the adjoint
$\Omega^*=\hat W^*(1\otimes\CJ)\tilde\G(\CJ\otimes\CJ)$.
The relations
\begin{align*}
(\CF_V \CJ\CF_V^*\,f)(q,\xi)&=|q|\,\Delta(q)^{-1/2}\,e^{i\langle \xi,\beta(q,q^{-1})\rangle}\,f\big(q^{-1},-{q^{-1}}^\flat\xi\big),\\
\big((\CF_V\otimes \CF_V)\hat W^* (\CF_V^*\otimes \CF_V^*)f\big)(q_1,\xi_1;q_2,\xi_2)
&= |q_2|_V\, e^{-i\langle\xi_1,\beta(q_2,q_2^{-1}q_1)\rangle}\,f\big(q_2^{-1}q_1,{q_2^{-1}}^\flat\xi_1;q_2,\xi_1+\xi_2\big)
\end{align*}
and  Lemma \ref{Gomega} yield:

\medskip
$\displaystyle \big((\CF_V\otimes\CF_V)\Omega^* (\CF_V^*\otimes\CF_V^*)f\big)(q_1,\xi_1;q_2,\xi_2)$
\begin{align*}
&=\big((\CF_V\otimes \CF_V)\hat W^* (1\otimes\CJ)\tilde\G(\CJ\otimes\CJ)(\CF_V^*\otimes \CF_V^*)f\big)(q_1,\xi_1;q_2,\xi_2)\\
&=|q_2|\,e^{-i\langle\xi_1,\beta(q_2,q_2^{-1}q_1)\rangle}\,\big((\CF_V\otimes \CF_V) (1\otimes\CJ)\tilde\G(\CJ\otimes\CJ)(\CF_V^*\otimes \CF_V^*)
f\big)(q_2^{-1}q_1,{q_2^{-1}}^\flat\xi_1;q_2,\xi_1+\xi_2)\\
&=|q_2|^2\,\Delta(q_2)^{-1/2}\,e^{-i\langle\xi_1,\beta(q_2,q_2^{-1}q_1)\rangle}\,e^{i\langle \xi_1+\xi_2,\beta(q_2,q_2^{-1})\rangle}\\
&\qquad\big((\CF_V\otimes \CF_V) \tilde\G(\CJ\otimes\CJ)(\CF_V^*\otimes \CF_V^*)
f\big)(q_2^{-1}q_1,{q_2^{-1}}^\flat\xi_1;q_2^{-1},-{q_2^{-1}}^\flat\xi_1-{q_2^{-1}}^\flat\xi_2)\\
&=|q_1q_2|\,\Delta(q_1q_2)^{-1/2}\, \Delta(\eta^{-1}(\xi_1))^{-1/2}\,
 \\
&\qquad e^{-i\langle\xi_1,\beta(q_2,q_2^{-1}q_1)\rangle}\,e^{i\langle \xi_1+\xi_2,\beta(q_2,q_2^{-1})\rangle}
\,e^{i\langle{q_2^{-1}}^\flat\xi_1,\beta(q_2^{-1}q_1,q_1^{-1})\rangle}b (\eta^{-1}(\xi_1+\xi_2)^{-1}q_2,q_2^{-1})\\
&\qquad \overline{b(\eta^{-1}(\xi_1)^{-1}q_2,q_2^{-1})}\,
b (\eta^{-1}(\xi_1)^{-1}q_2,q_2^{-1}\eta^{-1}( \xi_1))
\,\overline{b(\eta^{-1}(\xi_1+\xi_2)^{-1}q_2,q_2^{-1}\eta^{-1}(\xi_1))}\\
&\qquad\big((\CF_V\otimes \CF_V) (\CJ\otimes\CJ)(\CF_V^*\otimes \CF_V^*)f\big)(q_1^{-1},-{q_1^{-1}}^\flat\xi_1;q_2^{-1}\eta^{-1}(\xi_1),-{q_2^{-1}}^\flat\xi_2)\\
&=|\eta^{-1}(\xi_1)|\, e^{-i\langle\xi_1,\beta(q_2,q_2^{-1}q_1)\rangle}\,
 e^{i\langle \xi_1+\xi_2,\beta(q_2,q_2^{-1})\rangle}\\
&\qquad e^{i\langle{q_2^{-1}}^\flat\xi_1,\beta(q_2^{-1}q_1,q_1^{-1})\rangle}
\, e^{-i\langle {q_1^{-1}}^\flat\xi_1,\beta(q_1^{-1},q_1)\rangle}\,e^{-i\langle {q_2^{-1}}^\flat\xi_2,\beta(q_2^{-1}\eta^{-1}(\xi_1),\eta^{-1}(\xi_1)^{-1}q_2)\rangle}
 \\
&\qquad b (\eta^{-1}(\xi_1+\xi_2)^{-1}q_2,q_2^{-1})
 \overline{b(\eta^{-1}(\xi_1)^{-1}q_2,q_2^{-1})}\,
b (\eta^{-1}(\xi_1)^{-1}q_2,q_2^{-1}\eta^{-1}( \xi_1))
\\
&\qquad \overline{b (\eta^{-1}(\xi_1+\xi_2)^{-1}q_2,q_2^{-1}\eta^{-1}(\xi_1))}\,
f(q_1,\xi_1;\eta^{-1}(\xi_1)^{-1}q_2,{\eta^{-1}(\xi_1)^{-1}}^\flat\xi_2).
\end{align*}
Using the cocycle relation
$$
q_2\beta(q_2^{-1}q_1,q_1^{-1})-\beta(q_1,q_1^{-1})+\beta(q_2,q_2^{-1})-\beta(q_2,q_2^{-1}q_1)=0,
$$
we get that the above expression equals
\begin{multline*}
|\eta^{-1}(\xi_1)|
  e^{i\langle \xi_2,\beta(q_2,q_2^{-1})\rangle}e^{-i\langle {q_2^{-1}}^\flat\xi_2,\beta(q_2^{-1}\eta^{-1}(\xi_1),\eta^{-1}(\xi_1)^{-1}q_2)\rangle}
 \\
b (\eta^{-1}(\xi_1+\xi_2)^{-1}q_2,q_2^{-1})
 \overline{b(\eta^{-1}(\xi_1)^{-1}q_2,q_2^{-1})}\,
b (\eta^{-1}(\xi_1)^{-1}q_2,q_2^{-1}\eta^{-1}( \xi_1))
\\
\overline{b (\eta^{-1}(\xi_1+\xi_2)^{-1}q_2,q_2^{-1}\eta^{-1}(\xi_1))}\,
f(q_1,\xi_1;\eta^{-1}(\xi_1)^{-1}q_2,{\eta^{-1}(\xi_1)^{-1}}^\flat\xi_2).
\end{multline*}
Hence for the inverse operator, using the  identities  $q\beta(q^{-1},q)=\beta(q,q^{-1})$ and $\eta(\xi_1)^{-1}\eta^{-1}(\xi_1+
 \eta(\xi_1)^\flat\xi_2)=\eta^{-1}(\xi_2)$, we get the following:

\medskip
$\displaystyle \big((\CF_V\otimes\CF_V)\Omega (\CF_V^*\otimes\CF_V^*)f\big)(q_1,\xi_1;q_2,\xi_2)$
\begin{align*}
=\ &|\eta^{-1}(\xi_1)|^{-1}\,e^{i\langle {q_2^{-1}}^\flat\xi_2,\beta(q_2^{-1},q_2)-\beta(q_2^{-1}\eta^{-1}(\xi_1)^{-1},\eta^{-1}(\xi_1)q_2)\rangle}\\
&  \overline{b (q_2,q_2^{-1})}\,b (q_2,q_2^{-1}\eta^{-1}(\xi_1)^{-1})\,\overline{b(\eta^{-1}(\xi_2)^{-1}q_2,q_2^{-1}\eta^{-1}(\xi_1)^{-1})}\,b (\eta^{-1}(\xi_2)^{-1}q_2,q_2^{-1})\\
& f(q_1, \xi_1 ;\eta^{-1}(\xi_1)q_2,\eta^{-1}(\xi_1)^\flat\xi_2),
\end{align*}
from which the theorem follows by Lemma \ref{LDF}.
\end{proof}

\bigskip

\section{Proof of Theorem~\ref{thm:mult}}\label{ap:B}

Once we have the dual unitary $2$-cocycle $\Omega$, we can define a new product~$\star_\Omega$ on the Fourier algebra $A(G)$ of $G$ by
$$
(f_1\star_\Omega f_2)(g):=(f_1\otimes f_2)((\lambda_g\otimes\lambda_g)\Omega^*)
$$
and a representation $\pi_\Omega$ of $(A(G),\star_\Omega)$ on $L^2(G)$ by
$$
\pi_\Omega(f):=(f\otimes\iota)(\hat W\Omega^*).
$$
The weak operator closure of $\pi_\Omega(A(G))$ in $B(L^2(G))$ is a von Neumann algebra that is denoted by $W^*(\hat G;\Omega)$.

By \cite{BGNT3}*{Proposition~2.9} and its proof we have an isomorphism
\begin{equation}\label{eq:iso-Galois}
\Ad (\Op\circ\CJ)\colon W^*(\hat G;\Omega)\to B(L^2(Q)),
\end{equation}
where we view $B(L^2(Q))$ as an algebra of operators on $\HS(L^2(Q))$ acting by multiplication on the left. This isomorphism is $G$-equivariant with respect to the action $\Ad\pi$ on $B(L^2(Q))$ and the action $\Ad\rho$ on $W^*(\hat G;\Omega)$, where $\rho\colon G\to U(L^2(G))$ is the right regular representation of~$G$. We need to understand this isomorphism in terms of the operators~$\pi_\Omega(f)$ and~$\Op(f)$.

Define a product $\star$ on $L^2(G)$ by the identity
$$
\Op(f_1\star f_2)=\Op(f_1)\Op(f_2).
$$
An explicit formula for this product has essentially been obtained in Lemma~\ref{Gomega}. Namely, for $f_1,f_2\in\CF\CL_0(G)$, we have
\begin{align*}
\big(\CF_V(f_1\star f_2)\big)(q,\xi)=\ & b (\eta^{-1}(\eta(q)-\xi)^{-1},q)\,
\int_Q \overline{b (q_0^{-1},q)}\,b (q_0^{-1},q_0)\,\overline{b (\eta^{-1}(\eta(q)-\xi)^{-1},q_0)}\notag \\
&(\CF_Vf_1)(q,\eta(q)-\eta(q_0))\,(\CF_V f_2)(q_0, \eta(q_0)-\eta(q)+\xi)\,\frac{dq_0}{\,|q_0|},
\end{align*}
which we get from~\eqref{eq:star-product0} by letting $q_1=e$, $v_1=0$, $q_2=q$, $\xi_2=\xi$ and omitting the factor $\Delta(q_0)^{1/2}$ there, as we are now computing $\Op(f_1)\Op(f_2)$ instead of $\Op(f_1)D^{-1/2}\Op(f_2)$.

For $z\in\C$, define a linear operator $T_z$ on $\CF\CL_0(G)$ by
$$
(T_zf)(q,v)=\int_{\hat V}e^{i\langle \xi,v\rangle}\,(\CF_V f)(q,\xi)\,\frac{\Delta(q)^z}{\Delta(\eta^{-1}(\eta(q)-\xi))^z}\,d\xi.
$$
The role of this operator is expressed in the identity
\begin{equation}\label{eq:Deltaz}
\Delta^z(f_1\star (\Delta^{-z}f_2))=(T_zf_1)\star f_2\quad\text{for all}\quad f_1,f_2\in \CF\CL_0(G).
\end{equation}

\begin{proposition}
For every $f\in A(G)\cap\CF\CL_0(G)$, the isomorphism \eqref{eq:iso-Galois} maps $\pi_\Omega(f)$ into $\Op(T_{-1/2}f)$.
\end{proposition}

\bp
Similarly to \cite{BGNT3}*{Theorem~3.18}, this follows from identities \eqref{eq:Deltaz} and \cite{BGNT3}*{(3.20)}.
\ep

This proposition allows us to compare the GNS-maps for the canonical weights on $B(L^2(Q))$ and $W^*(\hat G;\Omega)$. Recall that on $B(L^2(Q))$ we consider the weight $\Tr(D^{1/2}\cdot D^{1/2})$, with the corresponding GNS-map $T\mapsto TD^{1/2}\in \HS(L^2(Q))$. On the other hand, we have a weight $\tilde\varphi$ on $W^*(\hat G;\Omega)$ defined by
$$
\tilde\varphi(T)1=\int_G (\Ad\rho(g))(T)\,dg\quad\text{for}\quad T\in W^*(\hat G;\Omega)_+.
$$
The canonical GNS-map $\tilde\Lambda\colon\mathfrak N_{\tilde\varphi}\to L^2(G)$ for this weight is determined by
\begin{equation}\label{eq:GNS-map}
\tilde\Lambda(\pi_\Omega(f))=\check f\quad\text{for}\quad f\in A(G)\quad\text{such that}\quad\check f\in L^2(G),
\end{equation}
where $\check f(g):=f(g^{-1})$. One can now easily check, cf.~\cite{BGNT3}*{Proposition 3.24}, that these two GNS-maps transfer to each other by the unitary $\Op\circ\CJ$, namely,
$$
\Op(\CJ\check f)=\Op(\Delta^{-1/2}f)=\Op(T_{-1/2}f)D^{1/2}
$$
for all $f\in\CF\CL_0(G)$. Since the modular conjugation for the weight $\Tr(D^{1/2}\cdot D^{1/2})$ is simply the map $T\mapsto T^*$ on $\HS(L^2(Q))$, we then get the following result that generalizes~\cite{BGNT3}*{Proposi\-tion~3.24}.

\begin{corollary}\label{cor:tildeJ}
The modular conjugation $\tilde J\colon L^2(G)\to L^2(G)$ for the canonical weight $\tilde\varphi$ on $W^*(\hat G;\Omega)$ and the GNS-map~\eqref{eq:GNS-map} is given by $\tilde J=\CJ \mathcal U \CJ J$, where $Jf:=\bar f$ and $\mathcal U$ is the unitary on $L^2(G)$ such that $\Op(f)^*=\Op(\mathcal U Jf)$ for $f\in L^2(G)$.
\end{corollary}

\begin{lemma}\label{lem:tildeJ}
The unitary $\mathcal U$ is given by
\begin{multline*}
(\CF_V \mathcal U f)(q,\xi)= \overline {b (\eta^{-1}(-{q^{-1}}^\flat\xi),\eta^{-1}(-{q^{-1}}^\flat\xi)^{-1})} \\
\overline{\psi(\eta^{-1}(\xi)^{-1},\eta^{-1}(\eta(q)-\xi),\eta^{-1}(-{q^{-1}}^\flat\xi)^{-1})} \,
(\CF_V f) \big(\eta^{-1}(\eta(q)-\xi),\xi\big).
\end{multline*}
\end{lemma}

\bp Denote by $\Op_0$ the map defined in the same way as $\Op$, but where we replace all $b$-factors by~$1$. In other words, if we define a unitary $\mathcal V$ on $L^2(G)$ by the identity
$$
(\CF_V\mathcal V f)(q,\xi)= b (q^{-1},q)
\,\overline{b (\eta^{-1}(\eta(q)-\xi)^{-1},q)}\,(\CF_{ V} f)(q,\xi),
$$
then $\Op(f)=\Op_0(\mathcal V f)$. Let $\mathcal U_0$ be the unitary operator defined by $\Op_0(f)^*=\Op_0(\mathcal U_0 Jf)$. Then
$\mathcal U =\mathcal V^*\mathcal U_0 J\mathcal VJ$.

A direct computation gives
$$
(\CF_V\mathcal U_0f)(q,\xi)=
(\CF_Vf)(\eta^{-1}(\eta(q)-\xi),\xi),
$$
cf.~\cite{BGNT3}*{Lemma 3.23}. Using that $(\CF_V Jf) (q,\xi)=\overline{\CF_Vf(q,-\xi)}$ we then get
\begin{align*}
(\CF_V\mathcal Uf)(q,\xi)=\ &\overline{ b(q^{-1},q)}
\,b(\eta^{-1}(\eta(q)-\xi)^{-1},q)\, \overline{ b (\eta^{-1}(\eta(q)-\xi)^{-1},\eta^{-1}(\eta(q)-\xi))}\\
& b(q^{-1},\eta^{-1}(\eta(q)-\xi))\, (\CF_Vf)(\eta^{-1}(\eta(q)-\xi),\xi).
\end{align*}

The final formula for $\CF_V\mathcal U$ follows from the following simplification of the $b$-factors: by identity~\eqref{ID}, we have
\begin{multline*}
\frac{b\big(q^{-1},\eta^{-1}(\eta(q)-\xi)\big)}{b(q^{-1},q)}
\frac{b\big(\eta^{-1}(\eta(q)-\xi)^{-1},q\big)}{b(\eta^{-1}(\eta(q)-\xi)^{-1},\eta^{-1}(\eta(q)-\xi))}\\
=
\frac{\psi(q^{-1},q,\eta^{-1}(-{q^{-1}}^\flat\xi))}{b(q,\eta^{-1}(-{q^{-1}}^\flat\xi))}
\frac{\psi(\eta^{-1}(\eta(q)-\xi)^{-1},\eta^{-1}(\eta(q)-\xi),\eta^{-1}(-{q^{-1}}^\flat\xi)^{-1})}{b(\eta^{-1}(\eta(q)-\xi),\eta^{-1}(-{q^{-1}}^\flat\xi)^{-1})}.
\end{multline*}
Using the coboundary relation $\psi=\partial b$ at $(q,\eta^{-1}(-{q^{-1}}^\flat\xi),\eta^{-1}(-{q^{-1}}^\flat\xi)^{-1})$, the last expression becomes
\begin{multline*}
\psi(q^{-1},q,\eta^{-1}(-{q^{-1}}^\flat\xi))\,\psi(\eta^{-1}(\eta(q)-\xi)^{-1},\eta^{-1}(\eta(q)-\xi),\eta^{-1}(-{q^{-1}}^\flat\xi)^{-1})\\
 \psi(q,\eta^{-1}(-{q^{-1}}^\flat\xi),\eta^{-1}(-{q^{-1}}^\flat\xi)^{-1})\,
\overline {b(\eta^{-1}(-{q^{-1}}^\flat\xi),\eta^{-1}(-{q^{-1}}^\flat\xi)^{-1})}
\end{multline*}
\begin{align*}
=\ &e^{i\langle \eta(q), \beta(q,\eta^{-1}(-{q^{-1}}^\flat\xi))+\beta(\eta^{-1}(\eta(q)-\xi),\eta^{-1}(-{q^{-1}}^\flat\xi)^{-1})
-q\beta(\eta^{-1}(-{q^{-1}}^\flat\xi),\eta^{-1}(-{q^{-1}}^\flat\xi)^{-1})\rangle}\\
&\qquad e^{-i\langle \xi,\beta(\eta^{-1}(\eta(q)-\xi),\eta^{-1}(-{q^{-1}}^\flat\xi)^{-1}) \rangle}\,
\overline {b(\eta^{-1}(-{q^{-1}}^\flat\xi),\eta^{-1}(-{q^{-1}}^\flat\xi)^{-1})}\\
=\ &\overline{\psi(\eta^{-1}( \xi)^{-1},\eta^{-1}(\eta(q)-\xi),\eta^{-1}(-{q^{-1}}^\flat\xi)^{-1})}\, \overline {b(\eta^{-1}(-{q^{-1}}^\flat\xi),\eta^{-1}(-{q^{-1}}^\flat\xi)^{-1})},
\end{align*}
where the last equality follows  from the cocycle identity for $\beta$  applied to the triple
$$
(q,\eta^{-1}(-{q^{-1}}^\flat\xi),\eta^{-1}(-{q^{-1}}^\flat\xi)^{-1}).
$$
\ep

\bp[Proof of Theorem~\ref{thm:mult}]
By \cite{DC}*{Proposition~5.4} and Corollary~\ref{cor:tildeJ}, we have  the formula
$$
\hat W_{\Omega}
=(\CJ \mathcal U \CJ J\otimes\hat J)\,\Omega\,\hat W^*\,(J\otimes\hat J)\,\Omega^*,
$$
where $(\hat J f)(g):=\Delta(g)^{-1/2}\overline{f(g^{-1})}$. The explicit expression for $(\CF_V\otimes\CF_V)\hat W_{\Omega} (\CF_V^*\otimes\CF_V^*)$ follows then by
a generalization of the computations given in the proof of \cite{BGNT3}*{Theorem 3.26}, one essentially just has to keep track of the three extra  phase factors coming from $\mathcal U$, $\Omega$ and $\Omega^*$.

In detail,  Theorem \ref{OmegaF} implies that for $f\in L^2(Q\times\hat V\times Q\times\hat V,|q|^{-1}dq\times d\xi\times |q|^{-1}dq\times d\xi)$ we have
\begin{multline*}
\big((\CF_V\otimes\CF_V)\Omega (\CF_V^*\otimes\CF_V^*)f\big)(q_1,\xi_1;q_2,\xi_2)=
 |\eta^{-1}(\xi_1)|^{-1}\,\overline {b(\eta^{-1}(\xi_2)^{-1},\eta^{-1}(\xi_1)^{-1})}\\
\overline{\psi(\eta^{-1}(\xi_2)^{-1},\eta^{-1}(\xi_1)^{-1},\eta^{-1}(\xi_1)q_2)}\,f(q_1, \xi_1 ;\eta^{-1}(\xi_1)q_2,\eta^{-1}(\xi_1)^\flat\xi_2),
\end{multline*}
hence also
\begin{multline*}
\big((\CF_V\otimes\CF_V)\Omega^* (\CF_V^*\otimes\CF_V^*)f\big)(q_1,\xi_1;q_2,\xi_2)=|\eta^{-1}(\xi_1)|\,
 b(\eta^{-1}(\xi_1+\xi_2)^{-1}\eta^{-1}(\xi_1),\eta^{-1}(\xi_1)^{-1})\\
\psi(\eta^{-1}(\xi_1+\xi_2)^{-1}\eta^{-1}(\xi_1),\eta^{-1}(\xi_1)^{-1},q_2)\, \,f\big(q_1, \xi_1 ;\eta^{-1}(\xi_1)^{-1}q_2,{\eta^{-1}(\xi_1)^{-1}}^\flat\xi_2\big).
\end{multline*}
 We also have
\begin{multline*}
\big((\CF_V\otimes \CF_V)\hat W^* (\CF_V^*\otimes \CF_V^*)f\big)(q_1,\xi_1;q_2,\xi_2)\\
= |q_2|\, e^{-i\langle\xi_1,\beta(q_2,q_2^{-1}q_1)\rangle}\,f\big(q_2^{-1}q_1,{q_2^{-1}}^\flat\xi_1;q_2,\xi_1+\xi_2\big)
\end{multline*}
and, for $f\in L^2(Q\times\hat V,|q|^{-1}dq\times d\xi)$,
\begin{align*}
(\CF_V \CJ\CF_V^*f)(q,\xi)=|q|\,\Delta(q)^{-1/2}\,e^{i\langle \xi,\beta(q,q^{-1})\rangle}\,f\big(q^{-1},-{q^{-1}}^\flat\xi\big).
\end{align*}
The last equality together with Lemma~\ref{lem:tildeJ} give us that

\medskip
$\displaystyle (\CF_V \CJ\CU\CJ\CF_V^*\,f)(q,\xi)$
\begin{align*}
&=|q|\,\Delta(q)^{-1/2}\,e^{i\langle \xi,\beta(q,q^{-1})\rangle}\,(\CF_V \CU\CJ\CF_V^*f)(q^{-1},-{q^{-1}}^\flat\xi)\\
&=|q|\,\Delta(q)^{-1/2}\,e^{i\langle \xi,\beta(q,q^{-1})+q\beta(q^{-1}\eta^{-1}(\xi),\eta^{-1}(\xi)^{-1})\rangle}\\
&\qquad \overline {b(\eta^{-1}(\xi),\eta^{-1}(\xi)^{-1})}\,(\CF_V \CJ\CF_V^*f)(q^{-1}\eta^{-1}(\xi),-{q^{-1}}^\flat\xi)\\
&=|\eta^{-1}(\xi)|\,\Delta(\eta^{-1}(\xi))^{-1/2}\,e^{i\langle \xi,\beta(q,q^{-1})+q\beta(q^{-1}\eta^{-1}(\xi),\eta^{-1}(\xi)^{-1})
-q\beta(q^{-1}\eta^{-1}(\xi),\eta^{-1}(\xi)^{-1}q)\rangle}\\
&\qquad \overline {b(\eta^{-1}(\xi),\eta^{-1}(\xi)^{-1})}\,f(\eta^{-1}(\xi)^{-1}q,{\eta^{-1}(\xi)^{-1}}^\flat\xi)\\
&=|\eta^{-1}(\xi)|\,\Delta(\eta^{-1}(\xi))^{-1/2}\,e^{i\langle \xi,\eta^{-1}(\xi)\beta(\eta^{-1}(\xi)^{-1},q)\rangle}\,\overline {b(\eta^{-1}(\xi),\eta^{-1}(\xi)^{-1})}\, f(\eta^{-1}(\xi)^{-1}q,{\eta^{-1}(\xi)^{-1}}^\flat\xi),
\end{align*}
where the last equality comes from the cocycle relation for $\beta$ at $(q^{-1}\eta^{-1}(\xi),\eta^{-1}(\xi)^{-1},q)$.
Therefore we get, using the identity $\eta^{-1}(-{\eta^{-1}(\xi)^{-1}}^\flat\xi)=\eta^{-1}(\xi)^{-1}$, that

\medskip
$\displaystyle \big((\CF_V\otimes\CF_V)\hat W_\Omega (\CF_V^*\otimes\CF_V^*)f\big)(q_1,\xi_1;q_2,\xi_2)$
\begin{align*}
&= \big((\CF_V\otimes\CF_V)(\CJ \mathcal U \CJ J\otimes\hat J)\,\Omega\,\hat W^*\,(J\otimes\hat J)\,\Omega^* (\CF_V^*\otimes\CF_V^*)f\big)(q_1,\xi_1;q_2,\xi_2)\\
&=|\eta^{-1}(\xi_1)|\,\Delta(\eta^{-1}(\xi_1))^{-1/2}\,e^{i\langle \xi_1 ,\eta^{-1}(\xi_1)\beta(\eta^{-1}(\xi_1)^{-1},q_1)\rangle}\,
 \overline {b(\eta^{-1}(\xi_1),\eta^{-1}(\xi_1)^{-1})}\\
&\qquad \big((\CF_V\otimes\CF_V)( J\otimes\hat J)\,\Omega\,\hat W^*\,(J\otimes\hat J)\,\Omega^*
  (\CF_V^*\otimes\CF_V^*)f\big)(\eta^{-1}(\xi_1)^{-1}q_1,{\eta^{-1}(\xi_1)^{-1}}^\flat\xi_1;q_2,\xi_2)\\
&=|\eta^{-1}(\xi_1)|\,\Delta(\eta^{-1}(\xi_1))^{-1/2}\,|q_2|\,\Delta(q_2)^{-1/2}\, e^{i\langle \xi_1 ,\eta^{-1}(\xi_1)\beta(\eta^{-1}(\xi_1)^{-1},q_1)\rangle}\,e^{i\langle \xi_2,\beta(q_2,q_2^{-1})\rangle}\\
&\qquad\overline {b(\eta^{-1}(\xi_1),\eta^{-1}(\xi_1)^{-1})}\\
&\qquad \overline{ \big((\CF_V\otimes\CF_V)\Omega\,\hat W^*\,(J\otimes\hat J)\,\Omega^*
  (\CF_V^*\otimes\CF_V^*)f\big)(\eta^{-1}(\xi_1)^{-1}q_1,-{\eta^{-1}(\xi_1)^{-1}}^\flat\xi_1;q_2^{-1},{q_2^{-1}}^\flat\xi_2)}\\
&=|\eta^{-1}(\xi_1)|^2\,\Delta(\eta^{-1}(\xi_1))^{-1/2}\,|q_2|\,\Delta(q_2)^{-1/2}\\
&\qquad  e^{i\langle \xi_1 ,\eta^{-1}(\xi_1)\beta(\eta^{-1}(\xi_1)^{-1},q_1)\rangle}\,e^{i\langle \xi_2,\beta(q_2,q_2^{-1})+q_2\beta(\eta^{-1}(\xi_1),\eta^{-1}(\xi_1)^{-1}q_2^{-1})\rangle}\\
&\qquad\overline{b(\eta^{-1}(\xi_1),\eta^{-1}(\xi_1)^{-1})}\, b(\eta^{-1}({q_2^{-1}}^\flat\xi_2)^{-1},\eta^{-1}(\xi_1))\\
&\qquad\overline{ \big((\CF_V\otimes\CF_V)\hat W^*\,(J\otimes\hat J)\,\Omega^*
  (\CF_V^*\otimes\CF_V^*)f\big)}\\
&\qquad\qquad\qquad\qquad\overline{\big(\eta^{-1}(\xi_1)^{-1}q_1, -{\eta^{-1}(\xi_1)^{-1}}^\flat\xi_1 ;\eta^{-1}(\xi_1)^{-1}q_2^{-1},{\eta^{-1}(\xi_1)^{-1}}^\flat{q_2^{-1}}^\flat\xi_2\big)}\\
&=|\eta^{-1}(\xi_1)|\,\Delta(\eta^{-1}(\xi_1))^{-1/2}\,\Delta(q_2)^{-1/2}\\
&\qquad
  e^{i\langle \xi_1 ,\eta^{-1}(\xi_1)(\beta(\eta^{-1}(\xi_1)^{-1},q_1)-\beta(\eta^{-1}(\xi_1)^{-1}q_2^{-1},q_2q_1))\rangle}\,
  e^{i\langle \xi_2,\beta(q_2,q_2^{-1})+q_2\beta(\eta^{-1}(\xi_1),\eta^{-1}(\xi_1)^{-1}q_2^{-1})\rangle}\\
&\qquad\overline {b(\eta^{-1}(\xi_1),\eta^{-1}(\xi_1)^{-1})}\, b(\eta^{-1}({q_2^{-1}}^\flat\xi_2)^{-1},\eta^{-1}(\xi_1))\\
&\qquad\overline{ \big((\CF_V\otimes\CF_V)(J\otimes\hat J)\,\Omega^*
  (\CF_V^*\otimes\CF_V^*)f\big)}\\
&\qquad\qquad\qquad\qquad\overline{\big(q_2q_1, -q_2^\flat\xi_1 ;\eta^{-1}(\xi_1)^{-1}q_2^{-1},{\eta^{-1}(\xi_1)^{-1}}^\flat({q_2^{-1}}^\flat\xi_2-\xi_1)\big)}\\
&=|q_2|^{-1}\,  e^{i\langle \xi_1 ,\eta^{-1}(\xi_1)(\beta(\eta^{-1}(\xi_1)^{-1},q_1)-\beta(\eta^{-1}(\xi_1)^{-1}q_2^{-1},q_2q_1)
  +\beta(\eta^{-1}(\xi_1)^{-1}q_2^{-1},q_2\eta^{-1}(\xi_1)))\rangle}\\
&\qquad  e^{i\langle \xi_2,\beta(q_2,q_2^{-1})+q_2\beta(\eta^{-1}(\xi_1),\eta^{-1}(\xi_1)^{-1}q_2^{-1})
-\beta(q_2\eta^{-1}(\xi_1),\eta^{-1}(\xi_1)^{-1}q_2^{-1})\rangle}\\
&\qquad \overline {b(\eta^{-1}(\xi_1),\eta^{-1}(\xi_1)^{-1})}\, b(\eta^{-1}({q_2^{-1}}^\flat\xi_2)^{-1},\eta^{-1}(\xi_1))\\
&\qquad \big((\CF_V\otimes\CF_V)\Omega^*(\CF_V^*\otimes\CF_V^*)f\big)
 \big(q_2q_1, q_2^\flat\xi_1 ;q_2\eta^{-1}(\xi_1),\xi_2-q_2^\flat\xi_1\big)\\
&=|\eta^{-1}(\xi_1+\eta(q_2^{-1}))|\\
&\qquad  e^{i\langle \xi_1 ,\eta^{-1}(\xi_1)(\beta(\eta^{-1}(\xi_1)^{-1},q_1)-\beta(\eta^{-1}(\xi_1)^{-1}q_2^{-1},q_2q_1)
  +\beta(\eta^{-1}(\xi_1)^{-1}q_2^{-1},q_2\eta^{-1}(\xi_1)))\rangle}\\
&\qquad e^{-i\langle \xi_1,q_2^{-1}\eta^{-1}(q_2^\flat\xi_1)\beta(\eta^{-1}(q_2^\flat\xi_1)^{-1},q_2\eta^{-1}(\xi_1)\rangle}\\
&\qquad  e^{i\langle \xi_2,\beta(q_2,q_2^{-1})+q_2\beta(\eta^{-1}(\xi_1),\eta^{-1}(\xi_1)^{-1}q_2^{-1})
-\beta(q_2\eta^{-1}(\xi_1),\eta^{-1}(\xi_1)^{-1}q_2^{-1})\rangle}\\
&\qquad  e^{i\langle \xi_2,\eta^{-1}(q_2^\flat\xi_1)\beta(\eta^{-1}(q_2^\flat\xi_1)^{-1},q_2\eta^{-1}(\xi_1)\rangle}\\
&\qquad \overline{b(\eta^{-1}(\xi_1),\eta^{-1}(\xi_1)^{-1})}\, b(\eta^{-1}({q_2^{-1}}^\flat\xi_2)^{-1},\eta^{-1}(\xi_1))\,b(\eta^{-1}(\xi_2)^{-1}\eta^{-1}(q_2^\flat\xi_1),\eta^{-1}(q_2^\flat\xi_1)^{-1})\\
& \qquad f \big(q_2q_1, q_2^\flat\xi_1 ;\eta^{-1}(\xi_1+\eta(q_2^{-1}))^{-1}\eta^{-1}(\xi_1),{\eta^{-1}(q_2^\flat\xi_1)^{-1}}^\flat(\xi_2-q_2^\flat\xi_1)\big).
  \end{align*}

The coboundary relation $\psi=\partial b$ at $(q',q,q^{-1})$ and $(q^{-1},q,q^{-1})$ gives:
$$
\frac{b(q'q,q^{-1})}{b(q,q^{-1})}=\overline{\psi(q',q,q^{-1})}\,\overline {b(q',q)}\qquad\mbox{and}\qquad b(q,q^{-1})=\psi(q^{-1},q,q^{-1})\, b(q^{-1},q).
$$
Therefore, for the $b$-factors in the expression above we have
\begin{multline*}
\overline {b(\eta^{-1}({q_2^{-1}}^\flat\xi_2)^{-1}\eta^{-1}(\xi_1),\eta^{-1}(\xi_1)^{-1})}\,
b(\eta^{-1}(\xi_2)^{-1}\eta^{-1}(q_2^\flat\xi_1),\eta^{-1}(q_2^\flat\xi_1)^{-1})
\\
\overline{\psi(\eta^{-1}({q_2^{-1}}^\flat\xi_2)^{-1}\eta^{-1}(\xi_1),\eta^{-1}(\xi_1)^{-1},\eta^{-1}(\xi_1))}\,
\psi(\eta^{-1}(\xi_1)^{-1},\eta^{-1}(\xi_1),\eta^{-1}(\xi_1)^{-1})\\
=\overline {b(\eta^{-1}({q_2^{-1}}^\flat\xi_2)^{-1}\eta^{-1}(\xi_1),\eta^{-1}(\xi_1)^{-1})}\,
b(\eta^{-1}(\xi_2)^{-1}\eta^{-1}(q_2^\flat\xi_1),\eta^{-1}(q_2^\flat\xi_1)^{-1})\\
e^{-i\langle\xi_2,q_2\beta(\eta^{-1}(\xi_1),\eta^{-1}(\xi_1)^{-1})\rangle}.
\end{multline*}

Hence we get:

\medskip
$\displaystyle \big((\CF_V\otimes\CF_V)\hat W_\Omega (\CF_V^*\otimes\CF_V^*)f\big)(q_1,\xi_1;q_2,\xi_2)$
\begin{align}\label{eq:W1}
&=|\eta^{-1}(\xi_1+\eta(q_2^{-1}))|\notag\\
&\qquad  e^{i\langle \xi_1 ,\eta^{-1}(\xi_1)(\beta(\eta^{-1}(\xi_1)^{-1},q_1)-\beta(\eta^{-1}(\xi_1)^{-1}q_2^{-1},q_2q_1)
  +\beta(\eta^{-1}(\xi_1)^{-1}q_2^{-1},q_2\eta^{-1}(\xi_1)))\rangle}\notag\\
&\qquad e^{-i\langle \xi_1 , q_2^{-1}\eta^{-1}(q_2^\flat\xi_1)\beta(\eta^{-1}(q_2^\flat\xi_1)^{-1},q_2\eta^{-1}(\xi_1))\rangle}\notag\\
&\qquad  e^{i\langle \xi_2,\beta(q_2,q_2^{-1})+q_2\beta(\eta^{-1}(\xi_1),\eta^{-1}(\xi_1)^{-1}q_2^{-1})
-\beta(q_2\eta^{-1}(\xi_1),\eta^{-1}(\xi_1)^{-1}q_2^{-1})-q_2\beta(\eta^{-1}(\xi_1),\eta^{-1}(\xi_1)^{-1}) \rangle}\notag\\
&\qquad e^{i\langle \xi_2,\eta^{-1}(q_2^\flat\xi_1)\beta(\eta^{-1}(q_2^\flat\xi_1)^{-1},q_2\eta^{-1}(\xi_1))\rangle}\notag\\
&\qquad\overline {b(\eta^{-1}({q_2^{-1}}^\flat\xi_2)^{-1}\eta^{-1}(\xi_1),\eta^{-1}(\xi_1)^{-1})}\,
b(\eta^{-1}(\xi_2)^{-1}\eta^{-1}(q_2^\flat\xi_1),\eta^{-1}(q_2^\flat\xi_1)^{-1})\notag\\
&\qquad f\big(q_2q_1, q_2^\flat\xi_1 ;\eta^{-1}(\xi_1+\eta^{-1}(q_2^{-1}))^{-1}\eta^{-1}(\xi_1),{\eta^{-1}(q_2^\flat\xi_1)^{-1}}^\flat(\xi_2-q_2^\flat\xi_1)\big).
\end{align}

The cocycle identity for $\beta$ at $(q_2,\eta^{-1}(\xi_1),\eta^{-1}(\xi_1)^{-1}q_2^{-1})$ gives
$$
\beta(q_2,q_2^{-1})+q_2\beta(\eta^{-1}(\xi_1),\eta^{-1}(\xi_1)^{-1}q_2^{-1})
-\beta(q_2\eta^{-1}(\xi_1),\eta^{-1}(\xi_1)^{-1}q_2^{-1})=\beta(q_2,\eta^{-1}(\xi_1)).
$$
The cocycle identity for $\beta$ at $(q_2,\eta^{-1}(\xi_1),\eta^{-1}(\xi_1)^{-1})$ gives
$$
\beta(q_2,\eta^{-1}(\xi_1))-q_2\beta(\eta^{-1}(\xi_1),\eta^{-1}(\xi_1)^{-1}) =-\beta(q_2\eta^{-1}(\xi_1),\eta^{-1}(\xi_1)^{-1}).
$$
Finally, the cocycle identity for $\beta$ at $( \eta^{-1}(q_2^\flat\xi_1),\eta^{-1}(q_2^\flat\xi_1)^{-1},q_2\eta^{-1}(\xi_1))$ gives
\begin{multline*}
 \eta^{-1}(q_2^\flat\xi_1)\beta(\eta^{-1}(q_2^\flat\xi_1)^{-1},q_2\eta^{-1}(\xi_1))\\=
 \beta( \eta^{-1}(q_2^\flat\xi_1),\eta^{-1}(q_2^\flat\xi_1)^{-1})- \beta(\eta^{-1}(q_2^\flat\xi_1),\eta^{-1}(q_2^\flat\xi_1)^{-1}q_2\eta^{-1}(\xi_1)).
\end{multline*}
Hence the $e^{i\langle\xi_2,\cdots\rangle}$-factors in \eqref{eq:W1} become
\begin{equation}\label{eq:xi2}
e^{i\langle \xi_2,-\beta(q_2\eta^{-1}(\xi_1),\eta^{-1}(\xi_1)^{-1})+ \beta( \eta^{-1}(q_2^\flat\xi_1),\eta^{-1}(q_2^\flat\xi_1)^{-1})
- \beta(\eta^{-1}(q_2^\flat\xi_1),\eta^{-1}(q_2^\flat\xi_1)^{-1}q_2\eta^{-1}(\xi_1))\rangle}.
\end{equation}

The cocycle identity for $\beta$ at $(\eta^{-1}(\xi_1)^{-1}q_2^{-1},q_2\eta^{-1}(\xi_1),\eta^{-1}(\xi_1)^{-1}q_1)$ gives
\begin{multline*}
-\beta(\eta^{-1}(\xi_1)^{-1}q_2^{-1},q_2q_1)+\beta(\eta^{-1}(\xi_1)^{-1}q_2^{-1},q_2\eta^{-1}(\xi_1))\\
=\eta^{-1}(\xi_1)^{-1}q_2^{-1}\beta(q_2\eta^{-1}(\xi_1),\eta^{-1}(\xi_1)^{-1}q_1),
\end{multline*}
and thus the $e^{i\langle\xi_1,\cdots\rangle}$-factors become
$$
e^{i\langle \xi_1 ,\eta^{-1}(\xi_1)\beta(\eta^{-1}(\xi_1)^{-1},q_1)+q_2^{-1}\beta(q_2\eta^{-1}(\xi_1),\eta^{-1}(\xi_1)^{-1}q_1)
   -q_2^{-1}\eta^{-1}(q_2^\flat\xi_1)\beta(\eta^{-1}(q_2^\flat\xi_1)^{-1},q_2\eta^{-1}(\xi_1))\rangle}.
$$
The cocycle identity for $\beta$ at $(q_2\eta^{-1}(\xi_1),\eta^{-1}(\xi_1)^{-1},q_1)$ gives
\begin{multline*}
 q_2 \eta^{-1}(\xi_1)\beta(\eta^{-1}(\xi_1)^{-1},q_1)+\beta(q_2\eta^{-1}(\xi_1),\eta^{-1}(\xi_1)^{-1}q_1)\\ =
 \beta(q_2,q_1)+\beta(q_2\eta^{-1}(\xi_1),\eta^{-1}(\xi_1)^{-1}),
\end{multline*}
and thus the $e^{i\langle\xi_1,\cdots\rangle}$-factors give
\begin{equation}\label{eq:xi1}
e^{i\langle q_2^\flat\xi_1 ,\beta(q_2,q_1)+\beta(q_2\eta^{-1}(\xi_1),\eta^{-1}(\xi_1)^{-1})
   - \beta( \eta^{-1}(q_2^\flat\xi_1),\eta^{-1}(q_2^\flat\xi_1)^{-1})+ \beta(\eta^{-1}(q_2^\flat\xi_1),\eta^{-1}(q_2^\flat\xi_1)^{-1}q_2\eta^{-1}(\xi_1))\rangle}.
\end{equation}

Equality~\eqref{eq:W1} together with the expressions~\eqref{eq:xi2} and~\eqref{eq:xi1} for the $e^{i\langle\xi_2,\cdots\rangle}$- and $e^{i\langle\xi_1,\cdots\rangle}$-factors imply the result.
\ep

\bigskip

\end{document}